\documentclass[11pt]{amsart}
\usepackage{amsmath,amssymb,amsfonts,amsthm,amscd,amstext,amsxtra,amsopn,array,url,verbatim,mathrsfs,enumerate,anysize,enumitem}
\usepackage{graphicx}
\usepackage{amsmath,amssymb,amsfonts,amsthm,amssymb,amscd,url,amstext,amsxtra,amsopn
,txfonts}
\usepackage{verbatim}
\marginsize{2.25cm}{2.25cm}{2.25cm}{2.25cm}
\usepackage[dvipsnames,usenames]{xcolor}
\usepackage[colorlinks=true,urlcolor=Black,citecolor=Black,linkcolor=Black]{hyperref}
\usepackage{times}
\usepackage{mathabx}

\makeatletter
\@namedef{subjclassname@2010}{%
  \textup{2010} Mathematics Subject Classification}
\makeatother

\newtheorem*{rep@theorem}{\rep@title}
\newcommand{\newreptheorem}[2]{%
\newenvironment{rep#1}[1]{%
 \def\rep@title{#2 \ref{##1}}%
 \begin{rep@theorem}}%
 {\end{rep@theorem}}}
\makeatother

\usepackage{amsrefs}
\newenvironment{rezabib}
  {\bibdiv\biblist\setupbib}
  {\endbiblist\endbibdiv}

\def\setupbib{\catcode`@=\active}
\begingroup\lccode`~=`@
  \lowercase{\endgroup\def~}#1#{\gatherkey{#1}}
\def\gatherkey#1#2{\gatherkeyaux{#1}#2\gatherkeyaux}
\def\gatherkeyaux#1#2,#3\gatherkeyaux{\bib{#2}{#1}{#3}}

\newtheorem{thm}{Theorem}[section]
\newtheorem{thrm}[thm]{Theorem}
\newtheorem{prop}[thm]{Proposition}
\newtheorem{remarks}[thm]{Remarks}

\newtheorem{lem}[thm]{Lemma}
\newtheorem{cor}[thm]{Corollary}

\newcommand{\QQ}{\mathbb{Q}}

\newcommand{\ZZ}{\mathbb{Z}}

\newcommand{\norm}{{\mathrm{N}}}

\newcommand{\ringO}{{\mathfrak{O}}}
\newcommand{\fa}{{\mathfrak{a}}}
\newcommand{\fb}{{\mathfrak{b}}}

\newcommand{\res}{{\mathrm{res}}}
\newcommand{\fp}{{\mathfrak{p}}}
\newcommand{\fm}{{\mathfrak{m}}}

\makeatletter
\def\imod#1{\allowbreak\mkern7mu({\operator@font mod}\,\,#1)}
\makeatother

\begin{document}

\title{Value-distribution of cubic Hecke $L$-functions}

\thanks{Research of the first author is partially supported by NSERC. Research of the second author was partially supported by a PIMS postdoctoral fellowship.}

\date{\today}

\keywords{\noindent value-distribution, logarithm of $L$-functions, logarithmic derivative of $L$-functions, cubic characters}

\subjclass[2010]{11R42,  11M41.}

\author{Amir Akbary}
\author{Alia Hamieh}

\address{Department of Mathematics and Computer Science \\
        University of Lethbridge \\
        Lethbridge, AB T1K 3M4 \\
        Canada}
        \email{amir.akbary@uleth.ca}
 \address{Department of Mathematics and Statistics \\
        University of Northern British Columbia \\
        Prince George, BC V2N4Z9 \\
        Canada}
\email{alia.hamieh@unbc.ca}

\begin{abstract}
Let $k=\QQ(\sqrt{-3})$, and let $c\in \mathfrak{O}_k$ be a square free algebraic integer such that  $c\equiv 1 \imod{\langle9\rangle}$. Let $\zeta_{k(c^{1/3})}(s)$ be the Dedekind zeta function of the cubic field $k(c^{1/3})$ and $\zeta_k(s)$ be the Dedekind zeta function of $k$. For fixed real $\sigma>1/2$, we obtain asymptotic distribution functions $F_{\sigma}$ for the values of the logarithm and the logarithmic derivative of the Artin $L$-functions
\begin{equation*}
L_c(\sigma)= \frac{\zeta_{k(c^{1/3})}(\sigma)}{\zeta_k(\sigma)},
\end{equation*}
as $c$ varies. Moreover, we express the characteristic function of $F_{\sigma}$ explicitly as a product indexed by the prime ideals of $\mathfrak{O}_k$. As a corollary of our results, we establish the existence of an asymptotic distribution function for the error term of the Brauer-Siegel asymptotic formula for the family of number fields $\{k(c^{1/3})\}_{c}$. We also deduce a similar result for the Euler-Kronecker constants of this family. 
\end{abstract}

\maketitle
\section{INTRODUCTION}

Several mathematicians have studied the distribution of values of $L$-functions associated with quadratic Dirichlet characters in the half plane $\Re(s)>\frac12$. One of the earliest works on this topic was done by Chowla and Erd\H{o}s in 1953.
Let $d$ be an integer such that $d$ is not a perfect square and $d\equiv 0,1\imod 4$. Consider, for $\Re(s)>0$, the infinite sum \[L_{d}(s)=\sum_{n=1}^{\infty}\frac{\left(\frac{d}{n}\right)}{n^{s}}.\] The function $L_{d}(s)$ is the $L$-series  attached to the quadratic Dirichlet character determined by the Kronecker symbol $\left(\frac{d}{.}\right)$. 
In \cite{chowla-erdos} the authors studied the distribution of values of $L_{d}(s)$ in the half-line $\sigma>\frac34$ for varying $d$ proving the following theorem.
\begin{thrm}[\bf Chowla-Erd\H{o}s]
If $\sigma>3/4$, we have
$$\lim_{x\rightarrow \infty} \frac{\#\{0<d\leq x;~d\equiv 0, 1\imod{4}~{\rm and}~ L_d(\sigma)\leq z\}}{x/2}=G(z)$$
exists. Furthermore $G(0)=0$, $G(\infty)=1$, and $G(z)$, the distribution function, is a continuous and strictly increasing function of $z$.
\end{thrm}
An analogous result also holds for negative $d$ with $0<-d\leq x$.
From Dirichlet's class number formula we know that $L_d(1)>0$, so under the Generalized Riemann Hypothesis (GRH), for $\sigma>1/2$, we have that  $L_d(\sigma)>0$ for all $d$. Note that the above result, for $\sigma>3/4$, unconditionally establishes that $L_d(\sigma)>0$ for almost all $d\leq x$ (i.e. except a set of size $o(x)$). 

In 1970, Elliott revisited this problem (for $\sigma=1$) and strengthened Chowla-Erd\H{o}s' theorem. The following is proved in \cite{elliott-0} (see also \cite[Chapter 20]{elliott-book} for a detailed exposition of this theorem).

 \begin{thrm}[\bf Elliott]
{There is a distribution function $F(z)$ such that}
{ $$\frac{\#\{0<-d\leq x;~d\equiv 0, 1\imod{4}~{\rm and}~ \log{L_d(1)}\leq z\}}{x/2}=F(z)+O\left( \sqrt{\frac{\log\log{x}}{\log{x}}}  \right)$$}{holds uniformly for all real $z$, and real $x\geq 9$. $F(z)$ has a probability density, may be differentiated any number of times, and has the characteristic function} 
{ $$\varphi_F(y)=\prod_{p} \left( \frac{1}{p}+\frac{1}{2}\left(1-\frac{1}{p}\right) \left(1-\frac{1}{p}\right)^{-iy}+ \frac{1}{2} \left(1-\frac{1}{p} \right) \left(1+\frac{1}{p}\right)^{-iy} \right)$$}
{\small which belongs to the Lebesgue class $L^1(-\infty, \infty)$.}
\end{thrm}

The above theorem gives detailed information on the distribution function in Chowla-Erd\H{o}s' theorem for $\sigma=1$ (it is smooth, has a density function, and its characteristic function can be computed) and provides an explicit error term. Moreover, the proof relies on the novel idea of interpreting the problem as a probability problem on sums of independent random variables.  
This point of view was explored by Elliott in the 1970's for several functions with similar expressions (see \cite{elliott-01}, \cite{elliott-02}, and \cite{elliott} for arguments mod $2\pi$ and logarithm of absolute values of quadratic $L$-functions).

  In \cite{GS} Granville and Soundararajan initiated the study of the tail of the distribution of the values of $L_d(1)$ as $d$ varies over the fundamental discriminants. Prior to this work, Hattori and Matsumoto \cite{MR1670215} had studied the tail of the distribution of $\log \zeta(\sigma+it)$ for $\frac12<\sigma<1$. This line of research has been extensively explored by Lamzouri for various $L$-functions (see for example  \cite{lamzouri2}, \cite{lamzouri3}, and \cite{MR3378382}).

Recently Ihara and Matsumoto started a systematic study of the value-distribution of the logarithm and the logarithmic derivative of $L$-functions on the half-plane $\Re(s)>1/2$ (see for example \cite{I-M1} and \cite{I-M}). 
Their approach has its roots in classical results such as L\'{e}vy's continuity theorem and Jessen-Wintner theory of infinite convolutions of distribution functions.

Following the method devised in \cite{I-M}, Mourtada and Murty proved the following (see \cite[Theorem 2]{M-M}).
\begin{thrm}[{\bf Mourtada-Murty}]
\label{MM-theorem}
Let $\sigma>1/2$, and assume the GRH (the Generalized Riemann Hypothesis for $L_d(s)$). Let $\mathcal{F}(Y)$ denote the set of the fundamental discriminants in the interval $[-Y, Y]$ and let $N(Y)=\#\mathcal{F}(Y)$. Then, there exists a probability density function $M_\sigma$, such that 
\begin{equation*}
\lim _{Y\rightarrow\infty} \frac{1}{N(Y)} \#\{ d\in \mathcal{F}(Y); ~ \left({L_d^\prime}/{L_d}\right)(\sigma)\leq z\}= \int_{-\infty}^{z} M_\sigma(t) dt.
\end{equation*}
Moreover, the characteristic function $\varphi_{F_\sigma}(y)$ of the asymptotic distribution function $F_\sigma(z)=  \int_{-\infty}^{z} M_\sigma(t) dt$  is given by 
$$\varphi_{F_\sigma}(y)= \prod_{p}  \left(\frac{1}{p+1}+ \frac{p}{2(p+1)}\exp{\left (  -\frac{iy\log{p}}{p^\sigma-1} \right)}+ \frac{p}{2(p+1)} \exp{\left ( \frac{iy\log{p}}{p^\sigma+1}  \right)} \right).$$

\end{thrm}
Notice that if $d$ is a fundamental discriminant then 
\begin{equation}
\label{quadratic}
L_d(s)=\frac{\zeta_{\QQ(\sqrt{d})}(s)}{\zeta(s)}, 
\end{equation}
where $\zeta_{\QQ(\sqrt{d})}(s)$ is the Dedekind zeta function of $\QQ(\sqrt{d})$ and $\zeta(s)$ is the Riemann zeta function.

In this paper, we prove an analogue of the above theorem for the cubic characters in the number field setting without assuming the GRH. In fact, we believe that one can remove the GRH assumption in  Theorem \ref{MM-theorem} by applying an appropriate zero density theorem for $L$-functions of quadratic Dirichlet characters and following a line of argument similar to the one we describe in Section \ref{sec:zero-density}.

Let $k=\QQ(\sqrt{-3})$ and $\zeta_3=\exp\left(2\pi i/3\right)$. Then  $\ringO_k=\ZZ[\zeta_3]$ is the ring of integers of $k$. Let \[\mathcal{C}:=\left\{c\in\mathfrak{O}_k;~c\neq 1 \text{ is square free and } c\equiv 1 \imod{\langle9\rangle} \right\}.\] 
For $c\in \mathcal{C}$, let $\zeta_{k(c^{1/3})}(s)$ be the Dedekind zeta function of the cubic field $k(c^{1/3})$. In analogy with \eqref{quadratic}, we define
\begin{equation}
\label{lc}
L_c(s)= \frac{\zeta_{k(c^{1/3})}(s)}{\zeta_k(s)}.
\end{equation}
We know that $L_{c}(s)=L(s,\chi_c)L(s,\overline{\chi_{c}})$, where $L(s,\chi_{c})$ is the Hecke $L$-function associated with some cubic character $\chi_{c}$ of the $\langle c\rangle$-ray class group of $k$. We will see in Section \ref{sec:cubic} that, for $c\in \mathcal{C}$, $\chi_c$ can be identified with $\left(\frac{.}{c}\right)_3$, the cubic residue character modulo $c$.
We set
\[ \mathcal{L}_{c}(s)=\begin{cases} 
      \log L_{c}(s) & \text{(Case 1),} \\
       {L_c^\prime}/{L_c}(s) & \text{(Case 2),} 
   \end{cases}
\]  and

\[ \mathcal{L}(s,\chi_{c})=\begin{cases} 
      \log L(s,\chi_{c}) & \text{(Case 1),} \\
       {L^\prime}/{L}(s,\chi_{c}) & \text{(Case 2).} 
   \end{cases}
\]  

Notice that $ \mathcal{L}_{c}(s)=2 \Re\left(\mathcal{L}(s,\chi_{c})\right)$ when $s$ is a real number. The main results of this paper are two theorems that describe the distribution of the values $\mathcal{L}_{c}(s)$ and $\Re\left(\mathcal{L}(s,\chi_{c})\right)$ for a fixed $s$ as $c$ varies over $\mathcal{C}$.

\begin{thrm}\label{mainthrm}
Let $\sigma>\frac12$. Let $\mathcal{N}(Y)$ be the the number of elements $c\in \mathcal{C}$ with norm not exceeding $Y$. There exists a smooth density function $M_{\sigma}$ such that  
\[\lim_{Y\to\infty}\frac{1}{\mathcal{N}(Y)}\#\left\{{c}\in \mathcal{C}: \norm(c)\leq Y\;\; \text{and}\;\;\ \mathcal{L}_{c}(\sigma) \leq z \right\}=\int_{-\infty}^{z}M_{\sigma}(t)\;dt.\]
The asymptotic distribution function $F_\sigma(z)=  \int_{-\infty}^{z} M_\sigma(t)\; dt$ can be constructed as an infinite convolution over prime ideals $\fp$ of $k$,
\[F_{\sigma}(z)=\mbox{*}_{\fp}\;F_{\sigma,\fp}(z),\]
where
\[F_{\sigma,\fp}(z)=\begin{cases}\frac{1}{\norm(\fp)+1}\delta(z)+\frac{1}{3}\left(\frac{\norm(\fp)}{\norm(\fp)+1}\right)\sum_{j=0}^{2}\delta_{-a_{\fp,j}}(z) & \text{if}\; \fp\nmid\langle3\rangle,\\ \delta_{a_{\fp,0}}(z) & \text{if}\; \fp=\langle1-\zeta_{3}\rangle.\end{cases}\] 
Here $\delta_{a}(z):=\delta(z-a)$, $\delta$ is the Dirac distribution, and
\[a_{\fp,j}:=a_{\fp,j}(\sigma) =\begin{cases} 
    2 \Re\left(\log(1-\zeta_{3}^{j}\norm(\fp)^{-\sigma})\right)   & \text{ in (Case 1),} \\
   2\Re\left(\frac{\zeta_{3}^{j}  \log\norm(\fp) }{\norm(\fp)^{\sigma}-\zeta_{3}^{j}}\right)   & \text{in (Case 2).} 
   \end{cases}
   \]
Moreover, the density function $M_\sigma$ can be constructed as the inverse Fourier transform of the characteristic function $\varphi_{F_\sigma}(y)$, which in (Case 1) is given by  
 \begin{equation*}
\varphi_{F_\sigma}(y)=\exp\left(-2iy\log(1-3^{-\sigma})\right)\prod_{\fp\nmid\langle3\rangle}
\left( \frac{1}{\norm(\fp)+1}+\frac{1}{3}\frac{\norm(\fp)}{\norm(\fp)+1}\sum_{j=0}^{2}\exp\left(-2iy\log\left|1-\frac{\zeta_{3}^{j}}{\norm(\fp)^{\sigma}}\right|\right) \right)
\end{equation*}
and in (Case 2) is given by 
\begin{equation*}
\varphi_{F_\sigma}(y)=\exp\left(-2iy\frac{\log 3}{3^{\sigma}-1}\right)\prod_{\fp\nmid\langle3\rangle} \left( \frac{1}{\norm(\fp)+1}+\frac{1}{3}\frac{\norm(\fp)}{\norm(\fp)+1}\sum_{j=0}^{2}\exp\left(-2iy\Re\left(\frac{\zeta_{3}^{j}   \log\norm(\fp)    }{\norm(\fp)^{\sigma}-\zeta_{3}^{j}}\right)\right) \right).
\end{equation*}
 \end{thrm}

 {The value $\mathcal{L}_{c}(1)$
 has some arithmetic significance. By the class number formula we know that} $$L_c(1)= \frac{(2\pi)^2 \sqrt{3}h_c R_c}{\sqrt{|D_c|}},$$
{where $h_c$, $R_c$, and $D_c=(-3)^5(\norm(c))^2$ (see \cite[p.~427]{X}) are respectively the class number, the regulator, and the discriminant of the cubic extension $K_c=k(c^{1/3})$. Since the number fields $K_c$ all have a fixed degree (namely $6$) over $\mathbb{Q}$, then by the Brauer-Siegel theorem 
\begin{equation}
\label{BS}
\log{(h_c R_c)} \sim \log{{|D_c|^{1/2}}},
\end{equation}
as $\norm(c) \rightarrow \infty$. Thus,  one obtains the following distribution result regarding the error term of \eqref{BS}} as a consequence of Theorem \ref{mainthrm} (Case 1).
\begin{cor} Let $E(c)= \log{(h_c R_c)} - \log{{|D_c|^{1/2}}}$. Then
\[\lim_{Y\to\infty}\frac{1}{\mathcal{N}(Y)}\#\left\{{c}\in \mathcal{C}: \norm(c)\leq Y\;\; \text{and}\;\;\ E(c) \leq z  \right\}=\int_{-\infty}^{z+\log(4\sqrt{3} \pi^2)}M_{1}(t)\;dt,\]
where $M_1(t)$ is the smooth function described in Theorem \ref{mainthrm} (Case 1) for $\sigma=1$.
\end{cor}
{Another application is related to the Euler-Kronecker constant of a number field $K$ which is defined by}
$$\gamma_K= \lim_{s\rightarrow 1} \left(\frac{\zeta_K^\prime(s)}{\zeta_K(s)}+\frac{1}{s-1}  \right).$$
From \eqref{lc} one can see that
$$\frac{L_c^\prime(1)}{L_c(1)}= \gamma_{K_c}-\gamma_k.$$ 
{Since $\gamma_k$ is fixed, we have the following corollary of Theorem \ref{mainthrm} (Case 2).}
\begin{cor} There exists a smooth function $M_1(t)$ (as described in Theorem \ref{mainthrm} (Case 2) for $\sigma=1$) such that
\[\lim_{Y\to\infty}\frac{1}{\mathcal{N}(Y)}\#\left\{{c}\in \mathcal{C}: \norm(c)\leq Y\;\; \text{and}\;\;\ \gamma_{K_c}\leq z  \right\}=\int_{-\infty}^{z-\gamma_k}M_{1}(t)\;dt.\] 
\end{cor}

The following theorem provides information on the distribution of the values $\Re\left(\mathcal{L}(s,\chi_{c})\right)$.

 \begin{thrm}\label{thm:2}
 Let $s\in\mathbb{C}$ be such that $\Re(s)>\frac12$. 
 There exists a distribution function $F_{s}(z)$ such that
\[\lim_{Y\to\infty}\frac{1}{\mathcal{N}(Y)}\#\left\{{c}\in \mathcal{C}: \norm(c)\leq Y\;\; \text{and}\;\;\ 2\Re\left(\mathcal{L}(s,\chi_{c})\right) \leq z \right\}=F_{s}(z).\]
The asymptotic distribution function $F_s$ can be constructed as an infinite convolution $\mbox{*}_{\fp}\;F_{s,\fp}(z)$, over prime ideals $\fp$ of $k$, where $F_{s,\fp}(z)$ is defined with similar formulas as in 
Theorem  \ref{mainthrm}.
Moreover, the distribution function $F_s$ has the characteristic function $\varphi_{F_s}(y)$, which in (Case 1) is given by  
 \begin{equation*}
\varphi_{F_s}(y)= \exp\left(-2iy\log\left|1-3^{-s}\right|\right)\prod_{\fp\nmid\langle3\rangle}
\left( \frac{1}{\norm(\fp)+1}+\frac{1}{3}\frac{\norm(\fp)}{\norm(\fp)+1}\sum_{j=0}^{2}\exp\left(-2iy\log\left|1-\frac{\zeta_{3}^{j}}{\norm(\fp)^{s}}\right|\right) \right)
\end{equation*}
and in (Case 2) is given by 
\begin{equation*}
\varphi_{F_s}(y)= \exp\left(-2iy\Re\left(\frac{\log 3}{3^{s}-1}\right)\right)\prod_{\fp\nmid\langle3\rangle} \left( \frac{1}{\norm(\fp)+1}+\frac{1}{3}\frac{\norm(\fp)}{\norm(\fp)+1}\sum_{j=0}^{2}\exp\left(-2iy\Re\left(\frac{\zeta_{3}^{j}   \log\norm(\fp)    }{\norm(\fp)^{s}-\zeta_{3}^{j}}\right)\right) \right).
\end{equation*}

\end{thrm}

 \begin{remarks}\quad
 \begin{enumerate}[label=(\alph*)]
 \item \emph{The proof of both theorems follow very similar arguments. In this paper, we only show the proof of Theorem \ref{mainthrm}. We mention here that we are able to prove the existence of a smooth density function in Theorem \ref{mainthrm} by establishing a suitable upper bound for the characteristic function $\varphi_{F_s}(y)$ when $s$ is real (see Proposition \ref{mainprop2}).}
 
\item \emph{For $\sigma>1$, from expression \eqref{Lc} for $L_c(\sigma)$ we can deduce that}
\begin{equation}
\label{case2}
\log{L_c(\sigma)}=-2\log\left (1-3^{-\sigma}\right)-2\sum_{\fp \nmid \langle3\rangle} \log{\left| 1- \frac{\left(\frac{c}{\pi} \right)_3 }{\norm(\fp)^\sigma} \right|},
\end{equation}
\emph{where $\pi$ is the unique generator of the prime ideal $\fp$ with the property that $\pi \equiv 1 \imod{\langle3\rangle}$, and $\left(\frac{c}{.} \right)_3$ is the cubic residue symbol of $c$ modulo $\pi$. 
Thus, the values $\log{L_c(1)}$ can be modelled as a sum over $\fp\nmid \langle3\rangle$ of the random variables }
$$X_\fp=
\begin{cases}

 0& {\rm with~ probability}~ \frac{1}{\norm(\fp)+1},\\
   \log{\left| 1- \frac{\zeta_3 }{\norm(\fp)^\sigma} \right|}^{-2} &{\rm with~ probability}~ \frac{1}{3}\frac{\norm(\fp)}{\norm(\fp)+1},\\
 \log{\left| 1- \frac{\zeta_3^2 }{\norm(\fp)^\sigma} \right|}^{-2} &{\rm with~ probability}~ \frac{1}{3}\frac{\norm(\fp)}{\norm(\fp)+1},\\  
 \log{\left| 1- \frac{1 }{\norm(\fp)^\sigma} \right|}^{-2} &{\rm with~ probability}~ \frac{1}{3}\frac{\norm(\fp)}{\norm(\fp)+1}.
\end{cases}
$$
 \emph{The heuristic supporting the above probabilities is as follows. If $\fp=\langle \pi \rangle$ is a prime ideal and $c$ is a square free element, then $c$ lies in one of the $\norm (\fp^2)-1$ residue classes modulo $\fp^2$. Now since exactly $\norm(\fp)-1$ of these residue classes are multiples of $\fp$, then under the assumption of uniform distribution of $c$ in residue classes modulo $\fp^2$, we conclude that the probability that $\pi \mid c$ is $1/(\norm(\fp)+1)$. More precisely, we get}
 $$\lim_{Y\rightarrow \infty} \frac{\#\{\norm(c)\leq Y; c\in \mathcal{C}~{\rm and}~ \pi \mid c\}}{\#\{\norm(c)\leq Y; c\in \mathcal{C} \}}=\frac{1}{\norm(\fp)+1}. $$
\emph{Finally, under the assumption of the uniform distribution of the values of the cubic residue symbol $\left(\frac{c}{\pi} \right)_3$ among the elements $c$ with $(c, \pi)=1$, for $i=1, 2, 3$, we have}
 $$\lim_{Y\rightarrow \infty} \frac{\#\{\norm(c)\leq Y; c\in \mathcal{C}~{\rm and}~\left(\frac{c}{\pi} \right)_3=\zeta_3^i \}}{\#\{\norm(c)\leq Y; c\in \mathcal{C}~{\rm and}~ (c, \pi)=1 \}}=\frac{1}{3}\frac{\norm(\fp)}{\norm(\fp)+1}. $$
\emph{A similar heuristic for (Case 2) also applies by finding the derivative of both sides of \eqref{case2} and establishing the identity}
$$\frac{L_c^\prime(\sigma)}{L_c(\sigma)}=-\frac{2\log{3}}{3^\sigma-1} -2\sum_{\fp \nmid \langle3\rangle}  \Re\left( \frac{\left(\frac{c}{\pi} \right)_3 \log{\norm(\fp)}}{\norm(\fp)^\sigma-\left(\frac{c}{\pi} \right)_3 } \right).$$

\item \emph{Our results establish one dimensional distribution theorems for the values of some real functions
associated with extensions of number fields. A natural next step in this research is extending
such results to the corresponding complex-valued functions. In \cite{I-M1} and \cite{I-M} such two dimensional distribution
theorems (unconditionally in \cite{I-M1} and conditionally on GRH in \cite{I-M}) for the values of the logarithm or the logarithmic derivative of $L$-functions
on average over certain Dirichlet characters were established. It would be worthwhile to investigate the two dimensional distribution results for the function $\mathcal{L}(s,\chi)$ as $\chi$ varies over the family of characters considered in Theorem \ref{MM-theorem}, Theorem \ref{mainthrm}, and more generally the higher order characters studied in \cite{BGL}. }

\item \emph {If we let $m_k(\sigma)$ be the non-negative $k$-integral moments of $\mathcal{L}_{{c}}(\sigma)$, one may try to obtain the characteristic function of the distribution function $F_\sigma$  by computing $$\varphi_{F_\sigma}(y)= \sum_{k=0}^{\infty} \frac{m_k(\sigma)}{k!} (iy)^k,$$ 
 for $y$ such that the above series is absolutely convergent (see \cite[Lemma 1.44]{elliott-book-I}). The authors of \cite{CK} have applied this method for certain families of number fields $K$ to obtain results (some conditional) on the distribution of the values at $\sigma=1$ of the logarithm and logarithmic derivative of the Artin $L$-function $\frac{\zeta_K(\sigma)}{\zeta(\sigma)}$ (\cite[Corollaries 3.14 and 5.4]{CK}).
Our method has the advantage of representing the characteristic function as a convenient product on the range $\sigma > 1/2$.  }

\end{enumerate}
 \end{remarks}

Next we make some comments regarding the method of proof of Theorem \ref{mainthrm}. 
For a real-valued arithmetic function $f(n)$ and $y\in \mathbb{R}$, assume that
\begin{equation}\label{eqn:standard-sum}\lim_{N\rightarrow \infty} \frac{1}{N} \sum_{n\leq N} \exp(iyf(n))=\widetilde{M}(y),\end{equation}
where $\widetilde{M}(y)$ is continuous at $0$. Then it is known that $f$ possesses an asymptotic distribution function, i.e., 
$$\lim_{N\rightarrow \infty}\frac{\#\{n\leq N;~f(n) \leq z\}}{N}= F(z),$$
for all $z$ in which $F$ is continuous (see \cite[p. 432, Theorem 2.6]{T}). 
In this paper we consider a weighted version of \eqref{eqn:standard-sum}.

\begin{lem}
\label{mainlemma2}
Let $f$ be a real arithmetic function. Suppose that, as $N\rightarrow \infty$, the functions
$$\frac{\displaystyle{\sum_{n=1}^{\infty}} e^{iy f(n)}e^{-n/N}}{\displaystyle{\sum_{n=1}^{\infty}}e^{-n/N}} $$
converge point-wise on $\mathbb{R}$ to a function $\widetilde{M}(y)$ which is continuous at $0$. Then $f$ possesses a distribution function $F$. In this case, $\widetilde{M}$ is the characteristic function  of $F$. Moreover, if 
\begin{equation}
\label{decay}
\left|\widetilde{M}(y)\right|\leq \exp\left(-\eta\left|y\right|^{\gamma}\right),
\end{equation}
for some $\eta,\gamma>0$,
then $F(z)=\int_{-\infty}^{z}M(t)dt$ for a smooth function $M$, where
\begin{equation}
\label{inversion}
M(z)=(1/2\pi)\int_{\mathbb{R}}\exp\left(-izy\right)\widetilde{M}(y)dy.
\end{equation}
\end{lem}

In order to verify the conditions of the above lemma for the cases described in Theorem \ref{mainthrm}, we establish the following two propositions. 

\begin{prop}\label{mainprop1}
Let \[\mathcal{N}^{*}(Y)=\sum_{c\in\mathcal{C}}\exp(-\norm(c)/Y).\] Fix $\sigma>\frac12$ and $y\in\mathbb{R}$. Let $\widetilde{M}_\sigma(y)$ be the function given by one of the product formulas in Theorem \ref{mainthrm}. Then
\begin{equation*}
\lim_{Y\to\infty}\frac{1}{\mathcal{N}^{*}(Y)} \sum_{{{c}}\in \mathcal{C}}^{\star}\exp\left(iy\mathcal{L}_{{c}} (\sigma)\right)\exp(-\norm({c})/Y)=\widetilde{M}_{\sigma}(y),
\end{equation*} 
where $\star$ indicates that the sum is over $c$ such that $L_c(\sigma)\neq 0$. 
\end{prop} 
The proof of this proposition is long and involves several intricate computations. It is inspired by the method devised by Luo in \cite{L} and used by Xia in \cite{X}. The proof of Proposition \ref{mainprop1} for $1/2<\sigma \leq 1$ is delicate, and in this case the recent zero density theorem of Blomer,  Goldmakher, and Louvel \cite{BGL} plays a crucial role in establishing the result without assuming the GRH.

The other ingredient needed for the application of Lemma \ref{mainlemma2} in the proof of Theorem \ref{mainthrm} is the following. 
\begin{prop}\label{mainprop2}
Let $\delta>0$ be given, and fix $\sigma>\frac12$. For sufficiently large values of $y$, we have $$\left|\widetilde{M}_{\sigma}(y)\right|\leq\exp\left(-C|y|^{\frac1{\sigma}-\delta}\right),$$ where $C$ is a positive constant that depends only on $\sigma$ and $\delta$.
\end{prop}
Our proof of the above proposition is inspired by the method employed in \cite{M-M} for the case of quadratic $L$-functions which itself is based on the ideas in \cite{Wintner}.
\medskip\par

The structure of the paper is as follows. In Sections 2 and 3 we review some background materials on  distribution functions and cubic characters. We prove Proposition \ref{mainprop1} in Section 4. The proof of Proposition \ref{mainprop2} is given in Section 5. 
\bigskip\par
\noindent{\bf Notation.} Throughout the paper $\epsilon$ and $\epsilon^\prime$ denote arbitrary positive numbers, whereas $\epsilon_0$ denotes a fixed positive constant. The capital letters $F$ and $G$ are used for distribution functions and the letter $M$ represents a density function. 
The characteristic function of a distribution function $F$ is denoted by $\varphi_F$.  The Dedekind zeta function of a number field $K$ is denoted by $\zeta_K(s)$. We set $k=\mathbb{Q}(\sqrt{-3})$ and denote its ring of integers by $\mathfrak{O}_k=\mathbb{Z}[\zeta_3]$, where $\zeta_3=\exp(\frac{2\pi i}{3})$. Gothic letters represent ideals of $\mathfrak{O}_k$. The norm of an ideal $\fa$ is written as $\norm(\fa)$. We denote the multiplicity of a prime ideal $\mathfrak{p}$ in an ideal $\mathfrak{a}$ by  $\nu_\fp(\fa)$. 
The Hecke $L$-function associated with a character $\chi$ of $\mathfrak{f}$-ray class group of $k$  is denoted by $L(s, \chi)$. The arithmetic function $\mu(n)$ is the M\"{o}bius function. Finally, summations over square free elements (modulo units) of $\mathfrak{O}_k$ are denoted by ${\sum^{\flat}}$.\bigskip\par
\noindent{\bf Acknowledgements.} The authors would like to thank Youness Lamzouri for pointing out the possibility of removing GRH in our main result by employing a zero density estimate and for outlining a method for achieving this. The first author is grateful to Kumar Murty for helpful comments and discussion related to this work.
The first author would also like to thank Kirsty Chalker for producing the two figures in the paper. 

\section{DISTRIBUTIONS AND PROOF OF THEOREM \ref{mainthrm}}
We review some basic facts on asymptotic (limiting) distribution of real arithmetic functions. Our reference for this section is  Chapter III.2 of \cite{T}. We start with some notions from probability.

A function $F: \mathbb{R} \rightarrow [0, 1]$ is called a \emph{distribution function} if  $F$ is non-decreasing, right-continuous, and satisfies $F(-\infty)=0$ and $F(+\infty)=1$.  Since $F$ is non-decreasing its discontinuity set is countable. We denote the continuity set of a distribution function $F$ by $\mathcal{C}(F)$.  An example of a continuous distribution function is 
$$F(z)=\int_{-\infty}^{z} M(t) dt,$$
where $M(t)$ is a non-negative integrable function such that $\int_{-\infty}^{\infty} M(t) dt=1$. We call $M$ the \emph{density} function of $F$. 
The \emph{characteristic function} of the distribution function $F$ is the Fourier transform of the measure $dF(z)$. Thus,
$$\varphi_F(y):= \int_{-\infty}^{\infty} e^{iy z} dF(z).$$
Distribution functions are uniquely determined by their characteristic functions. In other words $F=G$ if and only if $\varphi_F=\varphi_G$ (see \cite[p. 430]{T}).
We say that a sequence $\{F_n\}_{n=1}^{\infty}$ of distribution functions \emph{converges weakly} to a function $F$ if we have $$\lim_{n\rightarrow \infty} F_n(z)=F(z),~~{\rm for~all~}z\in \mathcal{C}(F).$$  

 The celebrated theorem of L\'{e}vy establishes a relation between the weak convergence of distribution functions and the point-wise convergence of their corresponding characteristic functions. 
 
 \begin{thm}
 \label{Levy}
(i) (L\'{e}vy's continuity theorem)  Let $\{F_n\}_{n=1}^{\infty}$ be a sequence of distribution functions, and let $\{\varphi_{F_n}\}_{n=1}^{\infty}$ be the sequence of their characteristic functions. Then $F_n$ converges weakly to a distribution function $F$ if and only if $\varphi_{F_n}$ converges point-wise on $\mathbb{R}$ to a function $\varphi$ which is continuous at $0$. Furthermore, in this case, $\varphi$ is the characteristic function of $F$, and the convergence of $\varphi_{F_n}$ to $\varphi$ is uniform on any compact subset. 

(ii) In part (i) if $\varphi\in L^{1}$, there exists a continuous function $M$ such that $F(z)=\int_{-\infty}^{z} M(t)\;dt$.

(iii) In part (ii) if $y^{a}\varphi(y)\in L^{1}$ for every $0\leq a\leq k$, then $M\in C^{k}$. 
\end{thm}
 
\begin{proof}
(i) See \cite[p. 430, Theorem 2.4]{T}.

(ii) Since $\varphi\in L^{1}$, there exists a continuous function $M$ such that $dF(z)=M(z)\;dz$ (\cite[p.~347]{Bill}). 

(iii) This follows from \cite[Theorem 8.22 (d)]{Folland}. 
 \end{proof}
We say that $f$ possesses an \emph{asymptotic distribution function} $F$ if the sequence $\{F_{N}\}_{N=1}^{\infty}$ with
$$F_N(z)=\frac{\#\{n\leq N;~f(n) \leq z\}}{N} $$
converges weakly to a distribution function $F$, as $N\rightarrow \infty$.

We are now ready to prove the main lemma that describes the method used in the proofs of our theorems.

\begin{proof}[Proof of Lemma \ref{mainlemma2}]
We first observe that the function $$G_N(z)= \frac{\displaystyle{\sum_{n=1}^{\infty}} \psi_{f, z}(n)e^{-n/N}}{\displaystyle{\sum_{n=1}^{\infty}}e^{-n/N}},$$
where
$$\psi_{f, z}(n)=\begin{cases}
 1& \text{if}~ f(n)\leq z,\\
 0&\text{otherwise,}
 \end{cases}$$
 is a distribution function. Note that the characteristic function of $G_N$ is
 $$\varphi_{G_N}(y)=\int_{-\infty}^{\infty} e^{iy z} d{G_N}(z)=
 \frac{\displaystyle{\sum_{n=1}^{\infty}} e^{iy f(n)}e^{-n/N}}{\displaystyle{\sum_{n=1}^{\infty}}e^{-n/N}}.
 $$
Under the conditions of the lemma, by part (i) of Theorem \ref{Levy}, we have that $\{G_N\}$ converges weakly to a distribution function $F$. Thus,
 $$\lim_{N\rightarrow \infty} \frac{\displaystyle{\sum_{n=1}^{\infty}} \psi_{f, z}(n)e^{-n/N}}{\displaystyle{\sum_{n=1}^{\infty}}e^{-n/N}}=F(z)$$
 for all $z\in C(F)$. 
By an application of a Tauberian theorem of Hardy and Littlewood (see \cite[Theorem 98]{H}) we deduce that
 $$\lim_{N\rightarrow \infty}\frac{\#\{n\leq N;~f(n) \leq z\}}{N}= F(z),$$
 for all $z\in C(F)$. This establishes that $f$ possesses the distribution function $F$. 
 
Next if the upper bound \eqref{decay} holds for $\widetilde{M}(y)$, then by parts (ii) and (iii) of Theorem \ref{Levy} we have $F(z)=\int_{-\infty}^{z} M(t) dt$ for a smooth probability density function $M$. Since $M, \widetilde{M}\in L^1$ and $\widetilde{M}$ is the characteristic function of $F$, then    
 the Fourier inversion theorem \cite[Theorem 8.26]{Folland} implies \eqref{inversion}.
 \end{proof}

The \emph{convolution} of two distribution functions $F$ and $G$ is the distribution function $F*G$ defined by
$$(F*G)(z)=\int_{-\infty}^{\infty} F(z-y) dG(y)=\int_{-\infty}^{\infty} G(z-y) dF(y).$$
It can be shown that $\varphi_{F*G}=\varphi_F \varphi_G.$  We say that a distribution function $F$ is the infinite convolution of distribution functions $F_1, F_2, \ldots, F_n, \ldots$ if $F_1*F_2*\ldots *F_n$  converges weakly to $F$ as $n\rightarrow \infty$. In such case we write $F=*_i F_i$. The following theorem provides a necessary and sufficient condition for the existence of infinite convolutions.
\begin{thm}
\label{infinite-convolution}
The infinite convolution $*_i F_i$ exists if and only if there exists $\delta>0$ such that for $|y|\leq \delta$ we have
$$\lim_{m,n\rightarrow \infty} \prod_{m<j\leq n} \varphi_{F_j}(y)=1.$$
\end{thm}
\begin{proof}
See \cite[p. 434, Theorem 2.7]{T}.
\end{proof}

\begin{proof}[Proof of Theorem \ref{mainthrm}]
Applying Proposition \ref{mainprop1} gives \begin{equation*}
\lim_{Y\to\infty}\frac{1}{\mathcal{N}^{*}(Y)} \sum_{c\in \mathcal{C}}^{\star}\exp\left(iy\mathcal{L}_{c} (\sigma)\right)\exp(-\norm(c)/Y)=\widetilde{M}_{\sigma}(y).
\end{equation*}
In one of the steps in the proof of Proposition \ref{mainprop1}, it is shown (in Proposition \ref{newprop}) that  $\widetilde{M}_{\sigma}(y)$ has also a representation as an absolutely convergent Dirichlet series which is uniformly convergent for $y\in [-R, R]$. Therefore, $\widetilde{M}_{\sigma}(y)$ is continuous at $y=0$. Next 
by Proposition \ref{mainprop2}, we know that $\widetilde{M}_{\sigma}(y)$ satisfies \eqref{decay}.
Hence,  Lemma \ref{mainlemma2} establishes the existence of a smooth asymptotic distribution function $F_\sigma(z)=\int_{-\infty}^{z} M_\sigma(t) dt$, where $M_\sigma(t)=(1/2\pi)\int_{\mathbb{R}}\exp\left(-ity\right)\widetilde{M}_{\sigma}(y)dy$. Moreover, $\widetilde{M}_{\sigma}(y)=\varphi_{F_\sigma}(y)$. Since $\widetilde{M}_{\sigma}(0)=1$, there is $\delta>0$ such that $\widetilde{M}_{\sigma}(y)\neq 0$ for $|y|\leq \delta$. This shows that for $|y|\leq \delta$ we have
$$\lim_{m,n\rightarrow \infty} \prod_{m<j\leq n} \phi_{F_{\sigma, \fp_j}}(y)=1,$$
where $(\fp_j)$ is a sequence representing the prime ideals of $k$ and  $F_{\sigma, \fp_j}$ is defined in the statement of Theorem \ref{mainthrm}. Thus by Theorem \ref{infinite-convolution}, the infinite convolution of the distributions $F_{\sigma, \fp}$ exists and hence $F_{\sigma}(z)=\mbox{*}_{\fp}\;F_{\sigma,\fp}(z)$.
\end{proof}

\section{PRELIMINARIES ON  CHARACTERS}\label{sec:cubic}

Let $k=\QQ(\sqrt{-3})$, and let $\mathfrak{O}_{k}=\mathbb{Z}[\zeta_{3}]$ be its ring of integers, where $\zeta_{3}=\exp\left(\frac{2\pi i}{3}\right)$. We know that $\mathfrak{O}_{k}$ is a principal ideal domain and its  group of units is given by $U_k=\{\pm1,\pm\zeta_{3},\pm\zeta_{3}^2\}$. We also have $\langle3\rangle:=3\mathfrak{O}_{k}=\langle1-\zeta_{3}\rangle^2$, and $3$
is the only ramified prime in $\mathfrak{O}_{k}$. Moreover, any ideal $\mathfrak{a}$ in $\mathfrak{O}_k$ has a unique generator of the form $(1-\zeta_{3})^{r}a$ where $r\in\ZZ^{\geq0}$ and $a\equiv 1 \imod{\langle3\rangle}$. In particular, if $\mathfrak{a}$ is relatively prime to $\langle3\rangle$, then $\mathfrak{a}=\langle a\rangle$ with $a\equiv 1 \imod{\langle3\rangle}$.

\subsection{Cubic residue symbol}
Let us now recall the definition of the cubic residue symbol (character) and collect some related facts. The reader is referred to \cite[Chapter~9]{I-R} and \cite[Chapter~7]{lemmermeyer} for more details. Let $\pi\in\mathfrak{O}_{k}$ be a prime not dividing 3, and suppose that $\alpha\in\mathfrak{O}_{k}$ is not divisible by $\pi$. It follows that  $$\alpha^{(\norm(\pi)-1)/3}\equiv1,\zeta,\zeta_{3}^{2}\imod{\langle \pi \rangle}.$$ The cubic residue symbol of $\alpha$ modulo $\pi$, denoted $\left(\frac{\alpha}{\pi}\right)_{3}$, is the unique cube root of unity such that \[\alpha^{(\norm(\pi)-1)/3}\equiv\left(\frac{\alpha}{\pi}\right)_{3}  \imod{\langle \pi \rangle} \] We set $\left(\dfrac{\alpha}{\pi}\right)_{3}=0$ if $\pi\mid\alpha$. In fact, for $(\alpha, \pi)=1$ the congruence $x^{3}\equiv\alpha\imod{\langle\pi\rangle}$ is solvable in $\mathfrak{O}_{k}$ if and only if $\left(\dfrac{\alpha}{\pi}\right)_{3}=1$. Moreover, we have \begin{equation}\label{eqn:cub-res-props}\left(\frac{-1}{\pi}\right)_{3}=1,\quad\left(\frac{\alpha\beta}{\pi}\right)_{3}=\left(\frac{\alpha}{\pi}\right)_{3}\left(\frac{\beta}{\pi}\right)_{3},\quad\text{and}\quad \alpha\equiv\beta\imod{\langle \pi \rangle}\implies\left(\frac{\alpha}{\pi}\right)_{3}=\left(\frac{\beta}{\pi}\right)_{3}.\end{equation}

We next extend the definition of the cubic residue symbol to include symbols of the form $\left(\frac{\alpha}{\lambda}\right)_{3}$, where $\lambda$ is a non-prime element in $\mathfrak{O}_{k}$. Let $\alpha,\lambda\in\mathfrak{O}_{k}$ with $\lambda\not\equiv 0\imod{\langle1-\zeta_{3}\rangle}$. Since $\mathfrak{O}_{k}$ is a unique factorization domain, we set \[\left(\frac{\alpha}{\lambda}\right)_{3}=\begin{cases}1 & \text{if}\; \lambda\in U_{k},\\
 \left(\frac{\alpha}{\pi_{1}}\right)_{3}\cdots\left(\frac{\alpha}{\pi_{r}}\right)_{3} & \text{if}\; \lambda\notin U_{k}\;\text{and}\; \lambda=\pi_{1}\cdots\pi_{r},\;\text{with all}\; \pi_{i} \;\text{prime in}\;\mathfrak{O}_{k}.\end{cases}\] This definition retains properties analogous to the basic properties in (\ref{eqn:cub-res-props}).

An element $\alpha \in\mathfrak{O}_{k}$ is said to be primary if $\alpha\equiv \pm1 \imod{\langle3\rangle}$. The following proposition is the law of cubic reciprocity and its supplementary formula.
 \begin{prop}\label{prop:cub-rec}
 If $\alpha,\lambda\in\mathfrak{O}_{k}$ are primary, then  \[\left(\frac{\alpha}{\lambda}\right)_{3}=\left(\frac{\lambda}{\alpha}\right)_{3}.\]
  In addition, if  $\lambda=1+3a +3b\zeta_{3}$ for some $a,b\in\mathbb{Z}$, we get the supplementary formula 
 \[\left(\frac{\zeta_{3}}{\lambda}\right)_{3}=\zeta_{3}^{-(a+b)}\quad\quad\text{and}\quad\quad\left(\frac{1-\zeta_{3}}{\lambda}\right)_{3}=\zeta_{3}^{ a}.\]
 \end{prop}
 \begin{proof}
 This is \cite[Theorem~7.8]{lemmermeyer}.
 \end{proof}

The following large sieve type inequality for cubic residue symbols, established in \cite{HB}, plays an important role in the proof of Proposition \ref{mainprop1}.

\begin{lem}[\bf Heath-Brown]\label{HB}
For $\epsilon>0$ and $b_\alpha\in \mathfrak{O}_k$, we have
\begin{equation}\label{eqn:large-sieve}\sum_{\substack{\lambda\equiv1\imod{\langle3\rangle}\\\norm(\lambda)\leq M}}^{\flat}\left|\sum_{\substack{\alpha\equiv1\imod{\langle3\rangle}\\\norm(\alpha)\leq N}}^{\flat}b_{\alpha}~ \left(\frac{\alpha}{\lambda}\right)_3  \right|^2\ll_{\epsilon}(M+N+(MN)^{\frac{2}{3}})(MN)^{\epsilon}\sum_{\norm(\alpha)\leq N}^{\flat}|b_{\alpha}|^{2},\end{equation}
where $\displaystyle{\sum^{\flat}}$ denotes summation over square free elements (modulo units) of $\mathfrak{O}_k$. 
\end{lem}

\subsection{The class $\mathcal{C}$}
\label{C}
Recall the set $\mathcal{C}$ given in the introduction  by \[\mathcal{C}:=\left\{c\in\mathfrak{O}_k;~c\neq 1 \text{ is square-free and } c\equiv 1 \imod{\langle9\rangle} \right\}.\] 
The following estimate is used in Section \ref{sec:mean-value}. For a proof the reader is referred to \cite[p.~1194]{L}.
\begin{lem}\label{lem:luo-lemma}
As $Y\to\infty$, we have $$\sum_{\substack{c\in\mathcal{C}\\\gcd(\langle c \rangle),\fa)=1}}\exp\left(-\frac{\norm(c)}{Y}\right)=C_{\fa}Y+O_\epsilon(Y^{\frac12+\epsilon}\norm(\fa)^{\epsilon}),$$ where $$C_{\fa}=\frac{3}{4}\frac{\res_{s=1}\zeta_{k}(s)}{\left| H_{\langle9\rangle} \right|\zeta_{k}(2)}\prod_{\substack{\fp|\fa\\\fp\; \text{prime}}}\left(1+\norm(\fp)^{-1}\right)^{-1},$$ $\mathrm{res}_{s=1}(\zeta_{k}(s))$ denotes the residue of $\zeta_{k}(s)$ at $s=1$, and  $H_{\langle 9\rangle}$ denotes the $\langle 9 \rangle$-ray class group of $k$.
\end{lem}
As a corollary of this lemma we have
\begin{equation}
\label{N*}
\mathcal{N}^*(Y)= \sum_{\substack{c\in\mathcal{C}}}\exp\left(-\frac{\norm(c)}{Y}\right) \sim  \frac{3}{4}\frac{\res_{s=1}\zeta_{k}(s)}{\left| H_{\langle9\rangle} \right|\zeta_{k}(2)} Y,
\end{equation}
 and 
 \[\mathcal{N}(Y)= \#\left\{c\in\mathcal{C}:\norm(c)\leq Y\right\}\sim\frac{3\mathrm{res}_{s=1}(\zeta_{k}(s))}{4\left| H_{\langle 9\rangle}\right|\zeta_{k}(2)}Y,\]
 as $Y\rightarrow \infty$.

In this paper, we are mostly interested in the cubic residue symbols $\left(\dfrac{.}{c}\right)_{3}$ where $c\in\mathcal{C}$. By Proposition \ref{prop:cub-rec}, we have 
$$\left(\frac{u}{c}\right)_3=1,\quad\quad{\rm for\; all}\; u\in U_k\;{\rm and}\; c\in\mathcal{C}.$$ 
In other words, $\left(\dfrac{.}{c}\right)_{3}$ is trivial on the units of $\mathfrak{O}_k$. Thus, it can be viewed as a primitive character of the $\langle c\rangle$-ray class group of $k$,  and therefore $\left(\dfrac{.}{c}\right)_{3}$  can be identified with $\chi_c$, where $L_{c}(s)=L(s,\chi_c)L(s,\overline{\chi_{c}})$.  It also follows from Proposition \ref{prop:cub-rec} that \[\chi_{c}(\langle1-\zeta_{3}\rangle)=\left(\frac{1-\zeta_{3}}{c}\right)_{3}=1.\] Moreover, since any non-zero integral ideal $\mathfrak{a}$ in $\mathfrak{O}_k$ has a unique generator of the form $(1-\zeta_{3})^{r}a$ with $r\in\ZZ^{\geq0}$ and $a\equiv 1 \imod{\langle3\rangle}$, we have \[\chi_{c}(\mathfrak{a})=\left(\frac{a}{c}\right)_{3}.\] 
In fact, by the law of cubic reciprocity as stated in Proposition \ref{prop:cub-rec}, we have \[\chi_{c}(\mathfrak{a})=\left(\frac{a}{c}\right)_{3}=\left(\dfrac{c}{a}\right)_{3},\] for all non-zero integral ideals $\fa$ in $\mathfrak{O}_{k}$.

Now for $c\in \mathcal{C}$ and $\Re(s)>1$, let $$L(s, \chi_c)=\sum_{0\neq \mathfrak{a} \subset \mathfrak{O}_K} \frac{\chi_c(\mathfrak{a})}{\norm(\mathfrak{a})^s}$$
be  the Hecke $L$-function associated with $\chi_c$. Here $\fa$ varies over all non-zero ideals of $\mathfrak{O}_k$. This $L$-function has an analytic continuation to the whole complex plane and satisfies a functional equation that relates its values at $s$ to its values at $1-s$.
In view of the identification $\chi_{c}=\left(\dfrac{.}{c}\right)_{3}$ and applying the various properties of the cubic residue symbol, for $\Re(s)>1$, one gets 
\begin{equation*}
L(s, \chi_{c})= \frac{3^s}{3^s-1} \sum_{a \equiv 1 \imod{\langle3 \rangle} }\frac{\left(\frac{c}{a} \right)_3}{\norm( a )^s}.
\end{equation*}
This is especially useful in elucidating the analogy between our setting and the quadratic setting studied in \cite{elliott} and \cite{M-M}. Similarly, for $\Re(s)>1$, we get 
\begin{equation}
\label{Lc}
L_{c}(s)=\frac{\zeta_{k(c^{1/3})}(s)}{\zeta_k(s)}=L(s, \chi_c)L(s, {\bar{\chi}}_c)
= \frac{3^{2s}}{(3^s-1)^2} \sum_{\substack{a, b\\a \equiv 1 \imod{\langle 3 \rangle}\\ b \equiv 1 \imod{\langle3 \rangle}}}\frac{\left(\frac{c}{a} \right)_3   \widebar{\left(\frac{c}{b} \right)}_3 }{\norm( ab )^s}.
\end{equation}

Crucial to our work is the zero density theorem given in \cite[Corollary 1.6]{BGL} which allows us to prove our main theorem without assuming the GRH. The following is a statement of this theorem when applied to $L(s, \chi_c)$. However, the original result in \cite{BGL} applies in a more general setting to $L$-functions of $n$-th order Hecke characters.
 
 \begin{lem}[{\bf Blomer--Goldmakher--Louvel}]
 \label{lem:zer-density}
 For $\frac12< \sigma \leq  1$, $T\geq 1$ and $c\in\mathcal{C}$, let $N(\sigma,T, c)$
be the number of zeros $\rho =\beta +i\gamma$ of $L(s, \chi_{c})$ in the rectangle $\sigma\leq\beta\leq   1$ , $|\gamma|\leq T$ . Then
\[\sum_{\substack{c\in\mathcal{C}\\\norm(c)\leq Y}}N(\sigma,T,c)\ll Y^{g(\sigma)}T^{1+\frac{2-2\sigma}{3-2\sigma}}(YT)^{\epsilon},\]
where \[g(\sigma)=\begin{cases}
\frac{8(1-\sigma)}{7-6\sigma},& \frac{1}{2}<\sigma\leq \frac56,\\\frac{2(10\sigma-7)(1-\sigma)}{24\sigma-12\sigma^2-11} & \frac56< \sigma\leq1.\end{cases}\]
 \end{lem}
 The proof of Proposition \ref{mainprop1}, namely Section \ref{sec:eval-of-iii}, also requires a bound for $\exp\left(iy\mathcal{L}_c(s)\right)$ 
 to the right of the critical line (see Lemma \ref{lem:bound-exp} below). To establish such a bound, we need the following Lemma.
\begin{lem}\label{lem:bound-L'/L}
Let $t$ be a real number and suppose that $L(s,\chi_c)$ has no zero in the region $\Re(s)>\sigma_0$ and $|\Im(s)|\leq |t|+1$, then for any $\sigma_1>\sigma_0$, we have 
\[\frac{L'}{L}(\sigma_1+it,\chi_c)\ll \frac{\log\left(\norm(c)(|t|+2)\right)}{\sigma_1-\sigma_0}\]
\end{lem}
\begin{proof}
The result follows by adapting the proof of \cite[Lemma~2.2]{MR3378382} to the setting of the cubic Hecke characters $\chi_c$.
\end{proof}
\begin{lem}\label{lem:bound-exp} Let $\epsilon>0$ be given, and let $0<\lambda<\frac12$. Consider a positive number $\alpha_2$ such that $\alpha_2<1$. Suppose that $L(s,\chi_c)$ for $\norm(c)\leq Y$ has no zero in the region $1-\lambda-\epsilon\leq \Re(s)<1$, $|\Im(s)|\leq (\log Y)^A$ for some $A>0$. Then in the region $1-\alpha_2\lambda\leq \Re(s)<1$, $|\Im(s)|\leq (\log Y)^A-2$, we have \[\exp\left(iy\mathcal{L}_c(s)\right)\ll_{\epsilon',\epsilon}Y^{\epsilon'},\] for any $\epsilon'>0$.
\end{lem}

\begin{proof} In what follows, we shall suppress the dependence on $y$ and $A$ in our estimates. First observe that if $\mathcal{L}_c(s)=\log L_{c}(s)$, then it is easy to see that $\left|\exp\left(iy\mathcal{L}_c(s)\right)\right|\leq e^{\pi |y|}\ll1$ for all $s$. We will thus  assume that $\mathcal{L}_c(s)=\frac{L'_{c}}{L_c}(s)$. The proof closely follows the proof of \cite[Lemma~2]{luo}. Let $\alpha_1,\alpha_3,\alpha_4$ be positive numbers such that $1<1+\alpha_4\lambda<2$ and $\alpha_1<\alpha_2<\alpha_3<\alpha_4<1$. Set $s_0=2+it$ with $|t|\leq (\log Y)^A-2$. By the Borel-Caratheodory theorem \cite[Section~5.5,~p. 174]{MR3728294} applied on the circles $C_3=\{s:|s-s_0|=1+\alpha_3\lambda\}$ and $C_4=\{s:|s-s_0|=1+\alpha_4\lambda\}$, we have \begin{equation}\label{eqn:BC}\max_{s\in C_3}\left|\log\exp\left(iy\mathcal{L}_c(s)\right)\right|\leq \frac{2+\alpha_3\lambda}{(\alpha_4-\alpha_3)\lambda}\max_{s\in C_4}\log\left|\exp\left(iy\mathcal{L}_c(s)\right)\right|+\frac{2+(\alpha_3+\alpha_4)\lambda}{(\alpha_4-\alpha_3)\lambda}\left|\log\left(iy\mathcal{L}_c(s_0)\right)\right|.\end{equation}
In view of this inequality, we need to bound $\log\left|\exp\left(iy\mathcal{L}_c(s)\right)\right|$ on $C_4$. Recall that \[\mathcal{L}_c(s)=\frac{L'_c}{L_c}(s)=\frac{L'}{L}(s,\chi_c)+\frac{L'}{L}(s,\bar{\chi_c}).\] Hence, \[\log\left|\exp\left(iy\mathcal{L}_c(s)\right)\right|\ll \left|\frac{L'}{L}(s,\chi_c)\right|+\left|\frac{L'}{L}(s,\bar{\chi_c})\right|.\] By Lemma \ref{lem:bound-L'/L}, we know that for  $\sigma_1>1-\lambda$, $|t|\leq (\log Y)^A-1$ and $\norm(c)\leq Y$, we have \[\frac{L'}{L}(\sigma_1+it,\chi_c)\ll\frac{\log(\norm(c)(|t|+2))}{\sigma_1-(1-\lambda-\epsilon)}\ll_{\epsilon}\log Y.\] It follows that $\log\left|\exp\left(iy\mathcal{L}_c(\sigma_1+it)\right)\right|\ll _{\epsilon}\log Y$ for all $1-\alpha_4\lambda<\sigma_1$ and $|t|\leq (\log A)^A-2$. Applying this bound in \eqref{eqn:BC}, we get $|\log\exp\left(iy\mathcal{L}_c(\sigma_1+it)\right)|\ll_{\epsilon}\log Y$ for all $\sigma_1\geq 1-\alpha_3\lambda$ and $|t|\leq (\log Y)^A-2$. 

We repeat this argument with the circle $C_1=\{s:|s-s_0|=1-\alpha_1\lambda\}$ and another circle containing $C_1$ inside the region $\Re(s)>1$. This yields \[|\log\exp\left(iy\mathcal{L}_c(\sigma_1+it)\right)|\ll_{\epsilon} 1\] for all $\sigma_1\geq 1+\alpha_1\lambda$ and $|t|\leq (\log Y)^A-2$. In fact, given $\epsilon'>0$ and $Y\geq C(\epsilon')$ for some constant $C(\epsilon')$ depending on $\epsilon'$, we get \[|\log\exp\left(iy\mathcal{L}_c(\sigma_1+it)\right)|\ll_{\epsilon}\epsilon'\log Y.\]
Setting $M(r)=\max_{|s-s_0|=r}\left|\log\exp(iy\mathcal{L}_c(s))\right|$ and appealing to Hadamard's three circles theorem \cite[Section~5.3,~p.172]{MR3728294}, we get \[\log M(r_2)\leq \frac{\log(r_3/r_2)}{\log(r_3/r_1)}\log M(r_1)+\frac{\log(r_2/r_1)}{\log(r_3/r_1)}\log M(r_3)\] for $0<r_1<r_2<r_3$.  Letting $r_1=1-\alpha_1\lambda$, $r_2=2-\sigma_1$ and $r_3=1+\alpha_3\lambda$ with $1-\alpha_2\lambda\leq \sigma_1<1$ yields
\[M(r_2)\leq M(r_1)^{1-a}M(r_3)^a,\] where \[a=\frac{\log(r_2/r_1)}{\log(r_3/r_1)}\leq \frac{\log((1+\alpha_2\lambda)/(1-\alpha_1\lambda))}{\log((1+\alpha_3\lambda)/(1-\alpha_1\lambda))}<1.\] Hence, we have  \[|\log\exp\left(iy\mathcal{L}_c(\sigma_1+it)\right)|\ll_{\epsilon}(\epsilon')^{1-a}\log Y\] for all $1-\alpha_2\lambda\leq\sigma_1<1$ and $|t|\leq (\log Y)^A-2$. Therefore,  \[|\exp\left(iy\mathcal{L}_c(\sigma_1+it)\right)|\leq \exp\left(\left|\log\exp(iy\mathcal{L}_c(\sigma_1+it))\right|\right)\leq Y^{\left(C_{\epsilon}\right)(\epsilon')^{1-a}}\] for all $1-\alpha_2\lambda\leq \sigma_1<1$ and $|t|\leq (\log Y)^A-2$, where $C_\epsilon$ is a constant that depends on $\epsilon$. Since $\epsilon$ is fixed and $\epsilon'$ is arbitrary, the desired result follows.
\end{proof}

\subsection{A Polya-Vinogradov type inequality}

In the proof of Proposition \ref{mainprop1}, we employ the following  inequality for $\mathfrak{f}$-ray class characters of $k$ which is proved in \cite{HB-P}.
\begin{lem}[{\bf Heath-Brown}--{\bf Patterson}]\label{H-P}
Let $k=\mathbb{Q}(\sqrt{-3})$ and $\chi$ be a non-trivial character (not necessarily primitive) of $\mathfrak{f}$-ray class group of $k$. Then for $Y>1$ and $\epsilon>0$, we have
\begin{equation}\label{eqn:polya}\sum_{a\equiv 1 \imod{\langle 3 \rangle}}\chi(a)\exp(-\frac{\norm(a)}{Y})\ll_{\epsilon}\norm(\mathfrak{f})^{\frac{1}{2}+\epsilon}.
\end{equation}

\end{lem}

\section{PROOF OF PROPOSITION \ref{mainprop1}}

We start by finding a Dirichlet series representation for $\exp\left(iy\mathcal{L}_{c} (s)\right)$.
\begin{lem}\label{lem:exp(iyLc)}
Let $y\in\mathbb{R}$, and let $s\in\mathbb{C}$ be such that $\Re(s)>1$. Then $\exp\left(iy\mathcal{L}_{c}(s)\right)$ can be written as an absolutely convergent Dirichlet series as follows:
\[
\exp\left(iy\mathcal{L}_{c}(s)\right)=\sum_{\fa,\fb\subset\mathfrak{O}_{k}}\frac{\lambda_{y}(\fa)\lambda_{y}(\fb)\chi_{c}(\fa\fb^{2})}{\norm(\fa\fb)^{s}},
\]
where the function $\lambda_{y}$ is described in the proof.
\end{lem}
\begin{proof}
We prove this lemma following the discussion in \cite[Sections~1.2 and 3.2]{I-M}. In view of the identity $L_{c}(s)=L(s,\chi_{c})L(s,\overline{\chi}_{c})$, we get
\[\exp\left(iy\mathcal{L}_{c} (s)\right)=\begin{cases}\exp\left( iy\log L(s, \chi_c)\right)\exp\left(iy \log L(s, \overline{\chi}_c) \right) & \text{in (Case 1),}\\
\exp\left(iy \frac{L^\prime}{L} (s, \chi_c) \right)\exp\left( iy \frac{L^\prime}{L} (s, \overline{\chi}_c) \right) & \text{in (Case 2)}.
\end{cases}
\]

We start by providing an expression for $\exp\left( iy \frac{L^\prime}{L} (s, \chi_c) \right)$.
For $\Re(s)>1$, we have $$L(s,\chi_{c})=\prod_{\fp}(1-\chi_{c}(\fp)\norm(\fp)^{-s})^{-1}.$$ Taking the logarithmic derivative and exponentiating both sides give 
$$\exp\left( iy \frac{L^\prime}{L} (s, \chi_c) \right)=\prod_{\fp}\exp\left(-iy\frac{\chi_{c}(\fp)\log\left(\norm(\fp)\right)\norm(\fp)^{-s}}{1-\chi_{c}(\fp)\norm(\fp)^{-s}}\right).$$
Notice that one can write 
\begin{equation}
\label{G}
\exp\left(\frac{ut}{1-t}\right)=\sum_{r=0}^{\infty}G_{r}(u)t^r,
\end{equation}
for $u,t\in\mathbb{C}$ with $|t|<1$. In fact, the coefficients $G_{r}(u)$ in the above power series expansion are computed as $G_{0}(u)=1$ and for $r\geq1$ \[G_r(u)=\sum_{n=1}^{r} \frac{1}{n!} {{r-1}\choose{n-1}} u^n.\] 
Hence, we have
\begin{equation}\label{eqn:exp1}\exp\left( iy \frac{L^\prime}{L} (s, \chi_c) \right)=\prod_{\fp}\sum_{r=0}^{\infty}G_{r}\left(-iy\log\left(\norm(\fp)\right)\right)\left(\chi_{c}(\fp)\norm(\fp)^{-s}\right)^{r}.\end{equation}

Similar computations give
\begin{equation}\label{eqn:exp-case2}\exp\left( iy \log L(s, \chi_c) \right)=\prod_{\fp}\sum_{r=0}^{\infty}H_{r}\left(iy\right)\left(\chi_{c}(\fp)\norm(\fp)^{-s}\right)^{r},\end{equation} where $H_{r}(u)$ is given by
 \begin{equation}
 \label{H}
 \exp\left(-u\log(1-t)\right)=\sum_{r=0}^{\infty}H_{r}(u)t^{r},\quad\text{for } |t|<1.
 \end{equation}
 In fact, $H_{r}(u)$ can be explicitly computed as $H_{0}(u)=1$ and for $r\geq1$
\[H_{r}(u)=\frac{1}{r!}u(u+1)\cdots(u+r-1).\]
This gives rise to the arithmetic function $\lambda_{y}(\fa)$ defined on the integral ideals of $k$ as follows: \[\lambda_y(\fa)=\prod_{\mathfrak{p}} \lambda_y({\mathfrak{p}}^{\alpha_\mathfrak{p}})\quad\text{ and}\quad\lambda_y(\mathfrak{p} ^{\alpha_\mathfrak{p}} )=\begin{cases}H_{\alpha_\mathfrak{p}} \left(iy \right)&\text{in (Case 1),}\\G_{\alpha_\mathfrak{p}} \left(-iy\log{\norm(\mathfrak{p})}\right)&\text{in (Case 2)}.\end{cases} \]

It follows from equations (\ref{eqn:exp1}) and (\ref{eqn:exp-case2}) that
\begin{align*}\exp\left( iy \mathcal{L}_{c} (s) \right)&=\prod_{\fp}\sum_{r=0}^{\infty}\lambda_{y}(\fp^{r})\chi_{c}(\fp^{r})\norm(\fp^{r})^{-s}\prod_{\fp}\sum_{r=0}^{\infty}\lambda_{y}(\fp^{r})\overline{\chi}_{c}(\fp^{r})\norm(\fp^{r})^{-s}\\&=\sum_{\fa\subset\ringO_{K}}\lambda_{y}(\fa)\chi_{c}(\fa)\norm(\fa)^{-s}\sum_{\fb\subset\ringO_{K}}\lambda_{y}(\fb)\overline{\chi}_{c}(\fb)\norm(\fb)^{-s}\\&=\sum_{\fa,\fb\subset\ringO_{K}}\lambda_{y}(\fa)\lambda_{y}(\fb)\chi_{c}(\fa\fb^{2})\norm(\fa\fb)^{-s}\end{align*}
as desired.\end{proof}
It is shown in \cite[p.~92]{I-M} that the estimate
\begin{equation}\label{eqn:lambda-bound}
\lambda_{y}(\fa)\ll_{\epsilon,R}\norm(\fa)^{\epsilon}
\end{equation} holds for any $\epsilon>0$ and all $|y|\leq R$ where $R$ is a positive real number. This bound will be used in Section \ref{sec:mean-value}.

The proof of Proposition \ref{mainprop1} will be divided into two parts. First in Proposition \ref{newprop} we find a Dirichlet series representation for the limit considered in Proposition \ref{mainprop1}, and then in Proposition \ref{lem:euler} we find a product representation for this Dirichlet series.

The next five sections are dedicated to the proof of the following proposition. The proof of Proposition \ref{lem:euler} (the product formula) is given in Section \ref{sec:prod}.

\begin{prop}
\label{newprop}
Fix $\sigma=\frac12+\epsilon_0$ for some 
$\epsilon_0>0$. Then for all $y\in\mathbb{R}$ we have
\begin{equation*}
\label{main-limit}
\lim_{Y\to\infty}\frac{1}{\mathcal{N}^{*}(Y)} \sum_{c\in \mathcal{C}}^{\star}\exp\left(iy\mathcal{L}_{c} (\sigma)\right)\exp(-\norm(c)/Y)=\widetilde{M}_{\sigma}(y),
\end{equation*} where $\star$ means that the sum is over $c$ such that $L_c(\sigma)\neq 0$. The function
$\widetilde{M}_{\sigma}(y)$ is given by the absolutely convergent Dirichlet series \[\sum_{r_{1},r_{2}\geq0}\frac{\lambda_{y}(\langle1-\zeta_{3}\rangle^{r_1})\lambda_{y}(\langle1-\zeta_{3}\rangle^{r_2})}{3^{(r_1+r_{2})\sigma}}\sum_{\substack{\fa,\fb,\fm\subset\mathfrak{O}_{k}\\\gcd(\fa\fb\fm,\langle3\rangle)=1\\\gcd(\fa,\fb)=1}}\frac{\lambda_{y}(\fa^3\fm)\lambda_{y}(\fb^3\fm)}{\norm(\fa^3\fb^3\fm^2)^{\sigma}\displaystyle{\prod_{\substack{\fp|\fa\fb\fm\\\fp\; \text{prime}}}\left(1+\norm(\fp)^{-1}\right)}}.\]
\end{prop}

\subsection{Application of the zero density estimate}\label{sec:zero-density}
 For fixed $A>0$, let $R_{Y, \epsilon, A}$ be the rectangle with  the vertices $1+i(\log{Y})^{2A}$, $(1+\epsilon)/2+i(\log{Y})^{2A}$, $(1+\epsilon)/2-i(\log{Y})^{2A}$, and  $1-i(\log{Y})^{2A}$ (see Figure \ref{R}). 
 \begin{figure}
\begin{center}
\includegraphics[scale=0.8]{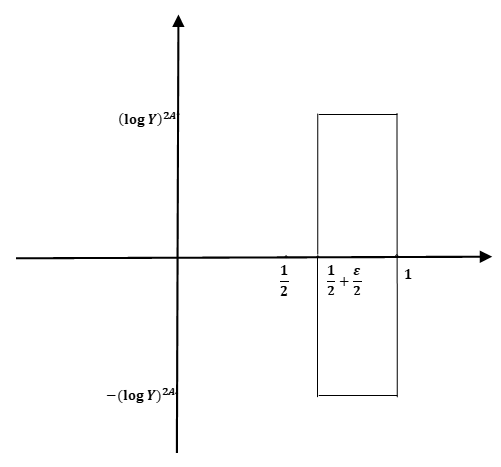}
\end{center}
\caption{The rectangle $R_{Y, \epsilon, A}$}
\label{R}
\end{figure}
We say that an element $c\in \mathcal{C}$ is in $\mathcal{Z}^{\mathrm{c}}$, if $L(s, \chi_c)$ does not have a zero in $R_{Y, \epsilon, A}$. Otherwise, it is in $\mathcal{Z}$. Thus, $\mathcal{C}=\mathcal{Z}\cup \mathcal{Z}^{\mathrm{c}}$. Note that $\mathcal{Z}$ and $\mathcal{Z}^{\mathrm{c}}$ depend on $Y$, $\epsilon$, and $A$.
\begin{lem}\label{lemma1}
For $\epsilon, A>0$, there is $\delta=\delta(\epsilon, A)<1$ such that 
$$\sum_{{c\in \mathcal{Z}}} \exp(-\norm(c)/Y) \ll Y^\delta.$$
\end{lem}
\begin{proof}
Let $N\left(\frac{1+\epsilon}{2}, (\log{Y})^{2A}, c\right)$ be the number of zeros of $L(s, \chi_c)$ in the rectangle $R_{Y, \epsilon, A}$. Then from Lemma \ref{lem:zer-density}
 we have
$$\sum_{\substack{{\norm(c)\leq Y}\\{{c\in \mathcal{Z}}}}} 1\leq  \sum_{\substack{{\norm(c)\leq Y}\\{c\in \mathcal{Z}}}} N\left(\frac{1+\epsilon}{2}, (\log{Y})^{2A}, c\right) \ll Y^\eta,  $$
for some $\eta=\eta(\epsilon, A)<1$. Hence, the Dirichlet series $\sum_{c\in \mathcal{Z}} 1/\norm(c)^s$ is absolutely convergent  (and thus analytic) for $\Re(s)>\eta$. We have
$$\sum_{c\in \mathcal{Z}} \exp(-\norm(c)/Y)=\frac{1}{2\pi i} \int_{(\eta+\epsilon^\prime)} \left( \sum_{c\in \mathcal{Z}}  \frac{1}{\norm(c)^s}\right) \Gamma(s) Y^s ds=O(Y^{\eta+\epsilon^\prime}).$$
Letting $\delta=\eta+\epsilon^\prime$ for sufficiently small $\epsilon^\prime$ implies the desired result.
\end{proof}
For $\sigma> \frac{1}{2}$ as fixed in Proposition \ref{newprop} and a sufficiently small $\epsilon>0$, we have
\begin{equation}
\label{main}
\sum_{c\in \mathcal{C}}^{\star}\exp\left(iy\mathcal{L}_{c} (\sigma)\right)\exp(-\norm(c)/Y)= \sum_{c\in \mathcal{Z}^{\mathrm{c}}}^{\star}+ \sum_{c\in \mathcal{Z}}^{\star}=\sum_{c\in \mathcal{Z}^{\mathrm{c}}}+ \sum_{c\in \mathcal{Z}}^{\star}.
\end{equation}
Observe that $\mathcal{L}_{c} (\sigma)$ is real, and so $|\exp\left(iy\mathcal{L}_{c} (\sigma)\right)|=1$. Thus, by Lemma \ref{lemma1}, we have
$$\sum_{c\in \mathcal{Z}}^{\star}\exp\left(iy\mathcal{L}_{c} (\sigma)\right)\exp(-\norm(c)/Y) \ll \sum_{c\in \mathcal{Z}}^{\star}\exp(-\norm(c)/Y)\ll Y^\delta.$$
Applying this estimate in \eqref{main} yields
\begin{equation}
\label{main1}
\sum_{c\in \mathcal{C}}^{\star}\exp\left(iy\mathcal{L}_{c} (\sigma)\right)\exp(-\norm(c)/Y)= \sum_{c\in \mathcal{Z}^{\mathrm{c}}} \exp\left(iy\mathcal{L}_{c} (\sigma)\right)\exp(-\norm(c)/Y)+ O(Y^\delta).
\end{equation}

For $c\in {\mathcal{Z}}^{\mathrm{c}}$, $y\in\mathbb{R}$, and $1/2<\sigma\leq1$, the next lemma gives a representation of $\exp\left(iy\mathcal{L}_{c} (\sigma)\right)$ as a sum of an infinite sum and a certain contour integral.
\begin{lem}
\label{rep}
 Let $\epsilon>0$ be given such that $\sigma\geq \frac{1}{2}+\epsilon$. Assume further that $\sigma\leq1$.
Then for $c\in \mathcal{Z}^{\mathrm{c}}$, we have
$$\exp\left(iy\mathcal{L}_{c} (\sigma)\right)= \sum_{\fa,\fb\subset\mathfrak{O}_{k}}\frac{\lambda_{y}(\fa)\lambda_{y}(\fb)\chi_{c}(\fa\fb^{2})}{\norm(\fa)^{{\sigma}}\norm(\fb)^{\sigma}}\exp\left(-\frac{\norm(\fa\fb)}{X}\right)-\frac{1}{2\pi i} \int_{L_{Y, \epsilon, A}} \exp\left(iy\mathcal{L}_{c} (\sigma+u)\right)\Gamma(u) X^u du, $$
where $L_{Y, \epsilon, A}=L_1+L_2+L_3+L_4+L_5$ (see Figure \ref{L}). Here $L_1$ is the vertical half-line given by $(1-\sigma+\frac{\epsilon}{2})+it(\log{Y})^A$ for $t\geq 1$, $L_2$ is the horizontal line segment given by $t+i(\log{Y})^A$ for $-\frac{\epsilon}{2} \leq t \leq (1-\sigma+\frac{\epsilon}{2})$, $L_3$ is the vertical line-segment given by  $-\frac{\epsilon}{2}+it(\log{Y})^A$ for $-1\leq t\leq 1$, $L_4$ is the horizontal line segment given by $t-i(\log{Y})^A$ for $-\frac{\epsilon}{2} \leq t \leq (1-\sigma+\frac{\epsilon}{2})$, and $L_5$ is the vertical half-line given by $(1-\sigma+\frac{\epsilon}{2})+it(\log{Y})^A$ for $t\leq -1$.
\end{lem}

\begin{figure}
\begin{center}
\includegraphics[scale=0.8]{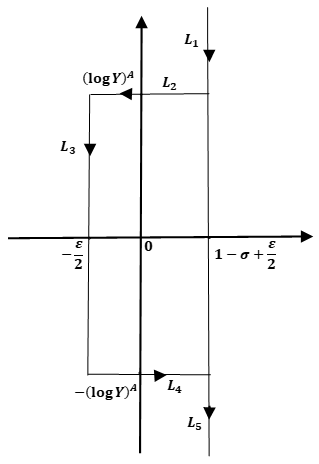}
\end{center}
\caption{The line segments $L_{Y, \epsilon, A}=L_1+L_2+L_3+L_4+L_5$}
\label{L}
\end{figure}

\begin{proof}
Using Lemma \ref{lem:exp(iyLc)} to write $\exp\left(iy\mathcal{L}_{c} (\sigma+u)\right)$, for $\Re(u)>1-\sigma$, as an absolutely convergent Dirichlet series, and applying the identity \begin{equation}\label{eqn:exponential}\exp\left( -\frac{1}{X}\right)=\frac{1}{2\pi i} \int_{(a)} \Gamma(u) X^u du,\end{equation}
for $a>0$, we have 
$$\frac{1}{2\pi i} \int_{(1-\sigma+\frac{\epsilon}{2})} \exp\left(iy\mathcal{L}_{c} (\sigma+u)\right) \Gamma(u) X^u du=  \sum_{\fa,\fb\subset\mathfrak{O}_{k}}\frac{{\lambda}_{y}(\fa)\lambda_{y}(\fb)\chi_{c}(\fa\fb^{2})}{\norm(\fa)^{{\sigma}}\norm(\fb)^{\sigma}}\exp\left(-\frac{\norm(\fa\fb)}{X}\right).$$
Moving the line of integration from $\Re(u)=1-\sigma+\frac{\epsilon}{2}$ to $L_{Y, \epsilon, A}$ and calculating the residue at $u=0$ yield the result. (Note that since $c\in \mathcal{Z}^{\mathrm{c}}$, the integrand has only a simple pole at $u=0$.) 
\end{proof}
Employing the above lemma in \eqref{main1} for $\sigma\leq1$ yields
\begin{equation}
\label{main2}
\sum_{c\in \mathcal{C}}^{\star}\exp\left(iy\mathcal{L}_{c} (\sigma)\right)\exp(-\norm(c)/Y)= (I)-(II)+(III)+ O(Y^\delta),
\end{equation}
where 
\begin{equation}
\label{one}
(I)=\sum _{c\in \mathcal{C}} \left ( \sum_{\fa,\fb\subset\mathfrak{O}_{k}}\frac{{\lambda}_{y}(\fa)\lambda_{y}(\fb)\chi_{c}(\fa\fb^{2})}{\norm(\fa)^{{\sigma}}\norm(\fb)^{\sigma}}\exp\left(-\frac{\norm(\fa\fb)}{X}\right)\right) \exp\left(-\frac{\norm(c)}{Y}\right),
\end{equation}
\begin{equation}
\label{two}
(II)= \sum _{c\in \mathcal{Z}} \left ( \sum_{\fa,\fb\subset\mathfrak{O}_{k}}\frac{{\lambda}_{y}(\fa)\lambda_{y}(\fb)\chi_{c}(\fa\fb^{2})}{\norm(\fa)^{{\sigma}}\norm(\fb)^{\sigma}}\exp\left(-\frac{\norm(\fa\fb)}{X}\right)\right) \exp\left(-\frac{\norm(c)}{Y}\right),
\end{equation}
and
\begin{equation}
\label{three}
(III)= \sum _{c\in \mathcal{Z}^{\mathrm{c}}} \left ( -\frac{1}{2\pi i} \int_{L_{Y, \epsilon, A}} \exp\left(iy\mathcal{L}_{c} (\sigma+u)\right)\Gamma(u) X^u du
\right) \exp\left(-\frac{\norm(c)}{Y}\right).
\end{equation}

\subsection{Evaluation of (I)}\label{sec:mean-value}

We prove an asymptotic formula for the sum $(I)$ given in \eqref{one}.
\begin{lem}\label{lem:luo-calcs}
Let
\[\widetilde{C}_{\sigma}(y)=\frac{3}{4}\frac{\res_{s=1}\zeta_{k}(s)}{\left| H_{\langle9\rangle} \right|\zeta_{k}(2)}\widetilde{M}_{\sigma}(y),\]
where $\widetilde{M}_{\sigma}(y)$ is defined in Proposition \ref{newprop}.
We have
\begin{equation*}
(I)=\widetilde{C}_{\sigma}(y)Y+O(YX^{\epsilon-\epsilon_0}+Y^{\frac{1}{2}+\epsilon}+Y^{\frac12+2\epsilon}X^{1-\epsilon_{0}+3\epsilon}),
\end{equation*}
for a sufficiently small $\epsilon>0$.
\end{lem}
\begin{proof}
From \eqref{one} we have
\begin{align*}(I)&=\sum_{\fa,\fb\subset\mathfrak{O}_{k}}\frac{\lambda_{y}(\fa)\lambda_{y}(\fb)}{\norm(\fa\fb)^{\sigma}}\exp\left(-\frac{\norm(\fa\fb)}{X}\right)\sum_{c\in\mathcal{C}}\chi_{c}(\fa\fb^{2})\exp\left(-\frac{\norm(c)}{Y}\right)\\&=\sum_{r_{1},r_{2}\geq0}\frac{\lambda_{y}(\langle1-\zeta_{3}\rangle^{r_{1}})\lambda_{y}(\langle1-\zeta_{3}\rangle^{r_2})}{3^{r_1\sigma+r_2\sigma}}\sum_{\substack{\fa,\fb\subset\mathfrak{O}_{k}\\\gcd(\fa\fb,\langle3\rangle)=1}}\frac{\lambda_{y}(\fa)\lambda_{y}(\fb)}{\norm(\fa\fb)^{\sigma}}\exp\left(-\frac{3^{r_{1}+r_{2}}\norm(\fa\fb)}{X}\right)\\&\hspace{.2in}\times\sum_{c\in\mathcal{C}}\chi_{c}(\fa\fb^{2})\exp\left(-\frac{\norm(c)}{Y}\right).
\end{align*}
Upon factoring out  $\fm=\gcd(\fa,\fb)$ in the above sum and noting that $\chi_c(\fm^3)=1$ for $\gcd(\langle c\rangle,\fm)=1$, we get
\begin{align}
\label{I}
(I)&=\sum_{r_{1},r_{2}\geq0}\frac{\lambda_{y}(\langle1-\zeta_{3}\rangle^{r_1})\lambda_{y}(\langle1-\zeta_{3}\rangle^{r_2})}{3^{r_1\sigma+r_2\sigma}}\sum_{\substack{\fa,\fb,\fm\subset\mathfrak{O}_{k}\\\gcd(\fa\fb\fm,\langle3\rangle)=1\\\gcd(\fa,\fb)=1}}\frac{\lambda_{y}(\fa\fm)\lambda_{y}(\fb\fm)}{\norm(\fa\fb)^{\sigma}\norm(\fm)^{2\sigma}}\exp\left(-\frac{3^{r_{1}+r_{2}}\norm(\fa\fb\fm^2)}{X}\right)\\&
\hspace{.2in}\times\sum_{\substack{c\in\mathcal{C}\\\gcd(\langle c\rangle,\fm)=1}}\chi_{c}(\fa\fb^{2})\exp\left(-\frac{\norm(c)}{Y}\right)\nonumber.\end{align} 

The contribution of cubes to the average expression \eqref{I} is given by 
\begin{align*}&\sum_{r_{1},r_{2}\geq0}\frac{\lambda_{y}(\langle1-\zeta_{3}\rangle^{r_1})\lambda_{y}(\langle1-\zeta_{3}\rangle^{r_2})}{3^{r_1\sigma+r_2\sigma}}\sum_{\substack{\fa,\fb,\fm\subset\mathfrak{O}_{k}\\\gcd(\fa\fb\fm,\langle3\rangle)=1\\\gcd(\fa,\fb)=1}}\frac{\lambda_{y}(\fa^3\fm)\lambda_{y}(\fb^3\fm)}{\norm(\fa\fb)^{3\sigma}\norm(\fm)^{2\sigma}}\exp\left(-\frac{3^{r_{1}+r_{2}}\norm(\fa^3\fb^3\fm^2)}{X}\right)\\&\hspace{.2in}\times\sum_{\substack{c\in\mathcal{C}\\\gcd(\langle c \rangle,\fa\fb\fm)=1}}\exp\left(-\frac{\norm(c)}{Y}\right).\end{align*}
Using Lemma \ref{lem:luo-lemma} and the estimate \eqref{eqn:lambda-bound} for $\lambda_y$ in the above expression, we get that the contribution of cubes to $(I)$ in \eqref{I} is \begin{align*}&Y\sum_{r_{1},r_{2}\geq0}\frac{\lambda_{y}(\langle1-\zeta_{3}\rangle^{r_1})\lambda_{y}(\langle1-\zeta_{3}\rangle^{r_2})}{3^{r_1\sigma+r_2\sigma}}\sum_{\substack{\fa,\fb,\fm\subset\mathfrak{O}_{k}\\\gcd(\fa\fb\fm,\langle3\rangle)=1\\\gcd(\fa,\fb)=1}}\frac{C_{\fa\fb\fm}\lambda_{y}(\fa^3\fm)\lambda_{y}(\fb^3\fm)}{\norm(\fa\fb)^{3\sigma}\norm(\fm)^{2\sigma}}\exp\left(-\frac{3^{r_{1}+r_{2}}\norm(\fa^3\fb^3\fm^2)}{X}\right)\\&\hspace{.2in}+O(Y^{\frac12+\epsilon}),\end{align*}
where $C_{\fa\fb\fm}$ is defined in Lemma \ref{lem:luo-lemma}.
 We can take $\epsilon$ to be sufficiently small so that $0<\epsilon < \epsilon_0$. Upon plugging the identity $$\exp\left(-\frac{3^{r_{1}+r_{2}}\norm(\fa^3\fb^3\fm^2)}{X}\right)=\int_{(1)}\Gamma(u)\left(\frac{3^{r_{1}+r_{2}}\norm(\fa^3\fb^3\fm^2)}{X}\right)^{-u}\;du$$
in the above expression and moving the line of integration to $\Re(u)=\epsilon-\epsilon_0$, we get that the contribution of cubes to $(I)$ is  \begin{align}\label{llast}Y\sum_{r_{1},r_{2}\geq0}\frac{\lambda_{y}(\langle1-\zeta_{3}\rangle^{r_1})\lambda_{y}(\langle1-\zeta_{3}\rangle^{r_2})}{3^{r_1\sigma+r_2\sigma}}\sum_{\substack{\fa,\fb,\fm\subset\mathfrak{O}_{k}\\\gcd(\fa\fb\fm,\langle3\rangle)=1\\\gcd(\fa,\fb)=1}}\frac{C_{\fa\fb\fm}\lambda_{y}(\fa^3\fm)\lambda_{y}(\fb^3\fm)}{\norm(\fa\fb)^{3\sigma}\norm(\fm)^{2\sigma}}+O(YX^{\epsilon-\epsilon_0})+O(Y^{\frac12+\epsilon}).\end{align}
Note that the coefficient of $Y$ in the above formula is the same as $\widetilde{C}_\sigma(y)$ given in the statement of the lemma.

We still need to estimate the contribution of non-cubes to $(I)$. This is accomplished following the methods of \cite{L} and \cite{X}.
In what follows, ${\sum^{\flat}}$ denotes summation over square free elements (modulo units) of $\mathfrak{O}_k$. 
For ideals $\fa$ and $\fb$ with $\gcd(\fa\fb,\langle3\rangle)=1$, we also write $\fa=\langle a\rangle$ and $\fb=\langle b\rangle$ for some uniquely determined elements $a,\;b\in\mathfrak{O}_k$ congruent to $1\imod{\langle3\rangle}$. With this notation in mind, we consider the character sum \begin{align*}
\sum_{c\in \mathcal{C}}\chi_{c}(\fa\fb^2)\exp\left(-\frac{\norm(c)}{Y}\right)&=\sum_{c\equiv1\imod{\langle9\rangle}}\chi_{c}(\fa\fb^2)\exp\left(-\frac{\norm(c)}{Y}\right)\sum_{\substack{d^{2}|c\\d\equiv1\imod{\langle3\rangle}}}\mu(d),
\end{align*}
which is by the cubic reciprocity, Proposition \ref{prop:cub-rec},
\begin{align*}
\\&=\sum_{c\equiv1\imod{\langle9\rangle}}\chi_{ab^2}(c)\exp\left(-\frac{\norm(c)}{Y}\right)\sum_{\substack{d^{2}|c\\d\equiv1\imod{\langle3\rangle}}}\mu(d).
\end{align*}
By choosing an auxiliary parameter $B$, we split the above sum into two parts to get
\begin{align*}\sum_{c\in \mathcal{C}}\chi_{c}(\fa\fb^2)\exp\left(-\frac{\norm(c)}{Y}\right)&=\sum_{\substack{d\equiv1\imod{\langle3\rangle}\\\norm(d)\leq B}}\mu(d)\chi_{ab^2}(d^2)\sum_{c\equiv\overline{d}^2\imod{\langle9\rangle}}\chi_{ab^2}(c)\exp\left(-\frac{\norm(d^2c)}{Y}\right)\\&+\sum_{d_1\equiv1\imod{\langle3\rangle}}\chi_{ab^2}(d_1^2)\sum_{\substack{d|d_1\\d\equiv1\imod{\langle3\rangle}\\\norm(d)>B}}\mu(d)\sum_{c\equiv\overline{d_1}^2\imod{\langle9\rangle}}^{\flat}\chi_{ab^2}(c)\exp\left(-\frac{\norm(d_1^2c)}{Y}\right).
\end{align*}

Let \begin{equation*}
R=\sum_{\substack{d\equiv1\imod{\langle3\rangle}\\\norm(d)\leq B}}\mu(d)\chi_{ab^2}(d^2)\sum_{c\equiv\overline{d}^2\imod{\langle9\rangle}}\chi_{ab^2}(c)\exp\left(-\frac{\norm(d^2c)}{Y}\right)
\end{equation*}
and 
\begin{equation*}
S=\sum_{d_1\equiv1\imod{\langle3\rangle}}\chi_{ab^2}(d_{1}^2)\sum_{\substack{d|d_1\\d\equiv1\imod{\langle3\rangle}\\\norm(d)>B}}\mu(d)\sum^{\flat}_{c\equiv\overline{d_1}^2\imod{\langle9\rangle}}\chi_{ab^2}(c)\exp\left(-\frac{\norm(d_1^2c)}{Y}\right),
\end{equation*} where $\overline{d}$ (respectively $\overline{d_1}$) is the multiplicative inverse of $d$ (respectively $d_1$) modulo $\langle 9 \rangle$. 

We shall estimate $R$ using Lemma \ref{H-P}. We have
\begin{align*}
R&=\sum_{\substack{d\equiv1\imod{\langle3\rangle}\\\norm(d)\leq B}}\mu(d)\chi_{ab^2}(d^2)\sum_{c\equiv\overline{d}^2\imod{\langle9\rangle}}\chi_{ab^2}(c)\exp\left(-\frac{\norm(d^2c)}{Y}\right)\\&=\frac{1}{\left|H_{\langle9\rangle}\right|}\sum_{\substack{d\equiv1\imod{\langle3\rangle}\\\norm(d)\leq B}}\mu(d)\chi_{ab^2}(d^2)\sum_{c\equiv1\imod{\langle3\rangle}}\chi_{ab^2}(c)\sum_{\chi\imod{\langle9\rangle}}\chi(d^2c)\exp\left(-\frac{\norm(d^2c)}{Y}\right)\\&=\frac{1}{\left|H_{\langle9\rangle}\right|}\sum_{\chi\imod{\langle9\rangle}}\sum_{\substack{d\equiv1\imod{\langle3\rangle}\\\norm(d)\leq B}}\mu(d)\chi_{ab^2}\chi(d^2)\sum_{c\equiv1\imod{\langle3\rangle}}\chi_{ab^2}\chi(c)\exp\left(-\frac{\norm(d^2c)}{Y}\right).
\end{align*}
Applying (\ref{eqn:polya}) of Lemma \ref{H-P} yields 
\begin{equation*}
R\ll B\norm(ab^2)^{\frac12+\epsilon}.
\end{equation*}
Therefore, the contribution of $R$ to $(I)$ in \eqref{I} is 
\begin{align*}
&\ll B\sum_{r_{1},r_{2}}\frac{\left|\lambda_{y}(\langle1-\zeta_{3}\rangle^{r_{1}})\right|\left|\lambda_{y}(\langle1-\zeta_{3}\rangle^{r_{2}})\right|}{3^{r_{1}\sigma+r_{2}\sigma}}\sum_{\fa,\fb,\fm}\frac{\norm(\fa\fb^{2})^{\frac{1}{2}+\epsilon}\left|\lambda_{y}(\fa\fm)\right|\left|\lambda_{y}(\fb\fm)\right|}{\norm(\fa)^{\sigma}\norm(\fb)^{\sigma}\norm(m)^{2\sigma}}\exp\left(-\frac{3^{r_{1}+r_{2}}\norm(\fa\fb\fm^{2})}{X}\right)\\&\ll B\sum_{\fa,\fb}\norm(\fa)^{-\sigma+\frac{1}{2}+2\epsilon}\norm(\fb)^{-\sigma+1+3\epsilon}\exp\left(-\frac{\norm(\fa\fb)}{X}\right)\\&= B\sum_{\fa}\norm(\fa)^{-\sigma+\frac12+2\epsilon}\sum_{\fb|\fa}\norm(\fb)^{\frac{1}{2}+\epsilon}\exp\left(-\frac{\norm(\fa)}{X}\right)\end{align*} which is at most 
\begin{align*}&B\sum_{\fa}\norm(\fa)^{-\sigma+1+4\epsilon}\exp\left(-\frac{\norm(\fa)}{X}\right)\ll BX^{2-\sigma+4\epsilon}.
\end{align*}

Now, let us analyze $S$. 
Recall that 
\begin{equation*}
S=\sum_{d_1\equiv1\imod{\langle3\rangle}}\chi_{ab^2}(d_{1}^2)\sum_{\substack{d|d_1\\d\equiv1\imod{\langle3\rangle}\\\norm(d)>B}}\mu(d)\sum^{\flat}_{c\equiv\overline{d_1}^2\imod{\langle9\rangle}}\chi_{ab^2}(c)\exp\left(-\frac{\norm(d_1^2c)}{Y}\right).
\end{equation*}
The contribution of $S$ to $(I)$ in \eqref{I} is 
\begin{align}\label{eqn:contribS}
&\sum_{r_{1},r_{2}\geq 0}\frac{\lambda_{y}(\langle1-\zeta_{3}\rangle^{r_{1}})\lambda_{y}(\langle1-\zeta_{3}\rangle^{r_{2}})}{3^{r_{1}\sigma+r_{2}\sigma}}\sum_{\substack{a\equiv1\imod{\langle3\rangle}\\b\equiv1\imod{\langle3\rangle}\\m\equiv1\imod{\langle3\rangle}\\(\langle a\rangle,\langle b\rangle)=1}}\frac{\lambda_{y}(\langle am\rangle)\lambda_{y}(\langle bm\rangle)}{\norm(a)^{\sigma}\norm(b)^{\sigma}\norm(m)^{2\sigma}}\exp\left(-\frac{3^{r_{1}+r_{2}}\norm(abm^2)}{X}\right)\\&\hspace{1em}\times\sum_{d_1\equiv1\imod{\langle3\rangle}}\chi_{ab^2}(d_{1}^2)\sum_{\substack{d|d_1\\d\equiv1\imod{\langle3\rangle}\\\norm(d)>B}}\mu(d)\sum^{\flat}_{c\equiv\overline{d_1}^2\imod{\langle9\rangle}}\chi_{ab^2}(c)\exp\left(-\frac{\norm(d_1^2c)}{Y}\right).\nonumber
\end{align}
We write $a=a_{1}a_{2}^{2}$, where $a_{1}$ is the square free part of $a$. Then (\ref{eqn:contribS}) becomes
\begin{align}\label{eqn:contribS2}
&\sum_{\substack{r_{1},r_{2}\geq0\\ a_{2},b,d_1,m}}\frac{\lambda_{y}(\langle1-\zeta_{3}\rangle^{r_{1}})\lambda_{y}(\langle1-\zeta_{3}\rangle^{r_{2}})\lambda_{y}(\langle bm\rangle)}{3^{r_{1}\sigma + r_{2}\sigma}\norm(a_{2})^{2\sigma}\norm(b)^{\sigma}\norm(m)^{2\sigma}}\sum_{\substack{d|d_1\\d\equiv1\imod{\langle3\rangle}\\\norm(d)>B}}\mu(d)\\&\hspace{1em}\times\sum^{\flat}_{a_{1}\equiv1\imod{\langle3\rangle}}\frac{\lambda_{y}(\langle a_{1}a_{2}^{2}m\rangle)}{\norm(a_{1})^{\sigma}}\chi_{a_{1}a_{2}^{2}b^{2}}(d_1^2)\exp\left(-\frac{3^{r_{1}+r_{2}}\norm(a_1a_{2}^{2}bm^2)}{X}\right)\sum^{\flat}_{c\equiv\overline{d_1}^{2}\imod{\langle9\rangle}}\chi_{a_{1}a_{2}^{2}b^{2}}(c)\exp\left(-\frac{\norm(cd_1^2)}{Y}\right).\nonumber
\end{align} 
Notice that we may assume that $\norm(a_1)\ll \frac{X^{1+\epsilon}}{\norm(a_2^2b)^{1+\epsilon}}$ to majorize \eqref{eqn:contribS2} by

\begin{align}\label{eqn:contribS2.5}
&\sum_{\substack{r_{1},r_{2}\geq0\\ a_{2},b,d_1,m}}\frac{\lambda_{y}(\langle1-\zeta_{3}\rangle^{r_{1}})\lambda_{y}(\langle1-\zeta_{3}\rangle^{r_{2}})\lambda_{y}(\langle bm\rangle)}{3^{r_{1}\sigma + r_{2}\sigma}\norm(a_{2})^{2\sigma}\norm(b)^{\sigma}\norm(m)^{2\sigma}}\sum_{\substack{d|d_1\\d\equiv1\imod{\langle3\rangle}\\\norm(d)>B}}\mu(d)\\&\hspace{1em}\times\sum^{\flat}_{\substack{a_{1}\equiv1\imod{\langle3\rangle}\\\norm(a_1)\ll\frac{X^{1+\epsilon}}{\norm(a_2^2b)^{1+\epsilon}}}}\frac{\lambda_{y}(\langle a_{1}a_{2}^{2}m\rangle)}{\norm(a_{1})^{\sigma}}\chi_{a_{1}a_{2}^{2}b^{2}}(d_1^2)\exp\left(-\frac{3^{r_{1}+r_{2}}\norm(a_{1}a_{2}^{2}bm^2)}{X}\right)\sum^{\flat}_{c\equiv\overline{d_1}^{2}\imod{\langle9\rangle}}\chi_{a_{1}a_{2}^{2}b^{2}}(c)\exp\left(-\frac{\norm(cd_1^2)}{Y}\right).\nonumber
\end{align} 
Applying Cauchy-Schwarz inequality gives that (\ref{eqn:contribS2.5}) is at most
\begin{align*}
&\sum_{\substack{r_{1},r_{2}\geq0\\ a_{2},b,d_1,m}}\left|\frac{\lambda_{y}(\langle1-\zeta_{3}\rangle^{r_{1}})\lambda_{y}(\langle1-\zeta_{3}\rangle^{r_{2}})\lambda_{y}(\langle bm\rangle)}{3^{r_{1}\sigma + r_{2}\sigma}\norm(a_{2})^{2\sigma}\norm(b)^{\sigma}\norm(m)^{2\sigma}}\sum_{\substack{d|d_1\\d\equiv1\imod{\langle3\rangle}\\\norm(d)>B}}\mu(d)\right|\\&\hspace{1em}\times\left(\sum^{\flat}_{\substack{a_{1}\equiv1\imod{\langle3\rangle}\\\norm(a_1)\ll\frac{X^{1+\epsilon}}{\norm(a_2^2b)^{1+\epsilon}}}}\left|\frac{\lambda_{y}(\langle a_{1}a_{2}^{2}m\rangle)}{\norm(a_{1})^{\sigma}}\right|^2\exp\left(-2\frac{3^{r_{1}+r_{2}}\norm(a_{1}a_{2}^{2}bm^2)}{X}\right)\right)^{\frac12}\nonumber\\&\hspace{1em}\times\left(\sum^{\flat}_{\substack{a_{1}\equiv1\imod{\langle3\rangle}\\\norm(a_1)\ll\frac{X^{1+\epsilon}}{\norm(a_2^2b)^{1+\epsilon}}}}\left|\sum^{\flat}_{c\equiv\overline{d_1}^{2}\imod{\langle9\rangle}}\chi_{a_{2}^{2}b^{2}}(c)\chi_{a_1}(c)\exp\left(-\frac{\norm(cd_1^2)}{Y}\right)\right|^2\right)^{\frac12},\nonumber
\end{align*} 
which is 
\begin{align}\label{eqn:contribS3}
&\ll \sum_{b,d_1,a_{2}}\left(\sum^{\flat}_{a_{1}\equiv1\imod{\langle3\rangle}}\norm(a_1)^{-2\sigma+\epsilon}\norm(a_{2})^{-4\sigma+\epsilon}\norm(b)^{-2\sigma+\epsilon}\norm(d_1)^{\epsilon}
\exp\left(-2\frac{\norm(a_{1}a_{2}^2b)}{X}\right)\right)^{\frac12}\\&\hspace{1em}\times\left(\sum^{\flat}_{\substack{a_{1}\equiv1\imod{\langle3\rangle}\\\norm(a_{1})\ll\frac{X^{1+\epsilon}}{\norm(a_{2}^2b)^{1+\epsilon}}}}\left|\sum^{\flat}_{\substack{c\equiv\overline{d_1}^{2}\imod{\langle9\rangle}\\\norm(c)\ll\frac{Y^{1+\epsilon}}{\norm(d_1^2)^{1+\epsilon}}}}\chi_{a_{1}}(c)\chi_{a_{2^2b^2}}(c)\exp\left(-\frac{\norm(cd_1^2)}{Y}\right)\right|^{2}\right)^{\frac12}.\nonumber
\end{align} 
Invoking the large sieve inequality (\ref{eqn:large-sieve}) of Lemma \ref{HB} in \eqref{eqn:contribS3}, we get that (\ref{eqn:contribS}) is 
\begin{align*}
&\ll\sum_{b,d_1,a_{2}}\left(\sum^{\flat}_{a_{1}\equiv1\imod{\langle3\rangle}}\norm(a_1)^{-2\sigma+\epsilon}\norm(a_{2})^{-4\sigma+\epsilon}\norm(b)^{-2\sigma+\epsilon}\norm(d_1)^{\epsilon}
\exp\left(-2\frac{\norm(a_{1}a_{2}^2b)}{X}\right)\right)^{\frac12}\\&\hspace{1em}\times\left(\frac{X^{1+\epsilon}}{\norm(a_{2}^2b)^{1+\epsilon}}+\frac{Y^{1+\epsilon}}{\norm(d_1^2)^{1+\epsilon}}+\left(\frac{XY}{\norm(a_{2}^2d_{1}^2b)}\right)^{\frac23(1+\epsilon)}\right)^{\frac12}\left(\frac{XY}{\norm(a_{2}^2d_{1}^2b)}\right)^{\frac{\epsilon}{2}(1+\epsilon)}
\left(\sum^{\flat}_{\norm(c)\ll\frac{Y^{1+\epsilon}}{\norm(d_{1}^2)^{1+\epsilon}}}\exp\left(-2\frac{\norm(cd_{1}^2)}{Y}\right)\right)^{\frac12}.
\end{align*}
Therefore, the contribution of $S$ to $(I)$ in \eqref{I} is at most 
\begin{align*}
&\sum_{b,d_{1},a_{2}}\left(\norm(a_{2})^{-4\sigma+\epsilon}\norm(b)^{-2\sigma+\epsilon}\norm(d_1)^{\epsilon}
\exp\left(-2\frac{\norm(b)}{X}\right)\right)^{\frac12}\\&\hspace{1em}\times\left(\frac{X^{\frac{1+\epsilon}2}}{\norm(a_{2})^{1+\epsilon}\norm(b)^{\frac{1+\epsilon}2}}+\frac{Y^{\frac{1+\epsilon}{2}}}{\norm(d_{1})^{1+\epsilon}}+\left(\frac{XY}{\norm(a_{2}^2d_{1}^2b)}\right)^{\frac{(1+\epsilon)}{3}}\right)\left(\frac{XY}{\norm(a_{2}^2d_{1}^2b)}\right)^{\frac{\epsilon}{2}(1+\epsilon)}
\left(\frac{Y}{\norm(d_{1}^2)}\right)^{\frac{1+\epsilon}2}\\&\ll (XY)^{\frac{\epsilon}{2}(1+\epsilon)}\sum_{b,d_{1}}\norm(b)^{-\sigma}\exp\left(-\frac{\norm(b)}{X}\right)\left(\frac{(XY)^{\frac{1+\epsilon}{2}}}{\norm(b)^{\frac12}\norm(d_1)}+\frac{Y^{1+\epsilon}}{\norm(d_1)^2}+\frac{X^{\frac{1+\epsilon}3}Y^{\frac56(1+\epsilon)}}{\norm(d_1)^{\frac53}\norm(b)^{\frac13}}\right).
\end{align*}
Notice that we may assume that  $B<\norm(d_1)<\sqrt{Y}$. Hence,
\begin{equation*}
\sum_{b,d_{1}}\norm(b)^{-\sigma-\frac12}\norm(d_1)^{-1}\exp\left(-\frac{\norm(b)}{X}\right)\ll Y^{\frac{\epsilon}{2}},
\end{equation*}
\begin{equation*}
\sum_{b,d_{1}}\norm(b)^{-\sigma}\norm(d_1)^{-2}\exp\left(-\frac{\norm(b)}{X}\right)\ll \frac{X^{1-\sigma}}{B},
\end{equation*}
and
\begin{equation*}
\sum_{b,d_{1}}\norm(b)^{-\sigma-\frac13}\norm(d_1)^{-\frac53}\exp\left(-\frac{\norm(b)}{X}\right)\ll X^{\frac23-\sigma}B^{-\frac23}.
\end{equation*}
The overall contribution of $S$ is 
\begin{equation*}
\ll (XY)^{\frac12+2\epsilon}+Y^{1+2\epsilon}\frac{X^{1-\sigma+\epsilon}}{B}+Y^{\frac56+\frac{3\epsilon}{2}}\frac{X^{1-\sigma+\epsilon}}{B^{\frac23}}.
\end{equation*}
Therefore, the contribution of $R+S$ to $(I)$ is 
\begin{equation*}
\ll BX^{2-\sigma+4\epsilon}+(XY)^{\frac12+2\epsilon}+Y^{1+2\epsilon}\frac{X^{1-\sigma+\epsilon}}{B}+Y^{\frac56+\frac{3\epsilon}{2}}\frac{X^{1-\sigma+\epsilon}}{B^{\frac23}}.
\end{equation*}
Upon taking $B=Y^{\frac12+\epsilon}X^{-\frac12-\frac{3}{2}\epsilon}$, this contribution becomes 
\begin{equation*}
\ll Y^{\frac12+2\epsilon}X^{\frac32-\sigma+3\epsilon}.
\end{equation*}
By combining this inequality with \eqref{llast}, we have 

\begin{equation}\label{one-estimate}
(I)=\widetilde{C}_{\sigma}(y)Y+O(YX^{\epsilon-\epsilon_0}+Y^{\frac{1}{2}+\epsilon}+Y^{\frac12+2\epsilon}X^{1-\epsilon_{0}+3\epsilon}).
\end{equation}
\end{proof}
\subsection{Evaluation of (II)}
In order to estimate $(II)$ in \eqref{two} for $\sigma\leq1$, we first provide an upper bound for the sum given in Lemma \ref{rep}.
\begin{lem}
\label{upper-bound}
Let $\sigma$ be as in Lemma \ref{rep}. For $\epsilon>0$ such that $\sigma\geq \frac{1}{2}+\epsilon$, we have
$$\sum_{\fa,\fb\subset\mathfrak{O}_{k}}\frac{{\lambda}_{y}(\fa)\lambda_{y}(\fb)\chi_{c}(\fa\fb^{2})}{\norm(\fa)^{{\sigma}}\norm(\fb)^{\sigma}}\exp\left(-\frac{\norm(\fa\fb)}{X}\right)=O\left(X^{1-\sigma+\epsilon} \right).$$
\end{lem}
\begin{proof}
Given $\epsilon>0$, we have
\begin{eqnarray*}
\sum_{\fa,\fb\subset\mathfrak{O}_{k}}\frac{{\lambda}_{y}(\fa)\lambda_{y}(\fb)\chi_{c}(\fa\fb^{2})}{\norm(\fa)^{{\sigma}}\norm(\fb)^{\sigma}}\exp\left(-\frac{\norm(\fa\fb)}{X}\right) &\ll& \sum_{\fa,\fb\subset\mathfrak{O}_{k}} \frac{1}{\norm(\fa\fb)^{\sigma-\frac{\epsilon}{2}}} \exp\left(-\frac{\norm(\fa\fb)}{X}\right) \\
&\ll& \sum_{\mathfrak{d}\subset\mathfrak{O}_k} \frac{1}{\norm(\mathfrak{d})^{\sigma-\epsilon}} \exp\left(-\frac{\norm(\mathfrak{d})}{X}\right)\\
&=&
\frac{1}{2\pi i} \int_{(c)} \zeta_k(u+\sigma-\epsilon) \Gamma(u) X^{u} du,\end{eqnarray*}
where $c+\sigma-\epsilon>1$. Moving the line of integration to the left of $1-\sigma+\epsilon$ yields the result.
\end{proof}
Now by using Lemmas \ref{lemma1} and \ref{upper-bound} in \eqref{two}, we have
\begin{equation}
\label{two-estimate}
(II)\ll Y^\delta X^{1-\sigma+\epsilon} .
\end{equation}

\subsection{Evaluation of (III)}\label{sec:eval-of-iii}
Note that $L_{Y, \epsilon, A}$ is the curve consisting of five line segments $L_1$, $L_2$, $L_3$, $L_4$, and $L_5$ defined in Lemma \ref{rep}. We need a version of Stirling's formula that states
$$|\Gamma(\alpha+it)|=O\left( \exp\left(- \frac{\pi}{2} |t| \right) |t|^{\alpha-\frac{1}{2} }\right),$$ 
where the constant is absolute for $\alpha$ in a closed bounded interval (see \cite[p. 176, Corollary 0.13]{T}).

Applying Lemma \ref{lem:bound-exp} with $\alpha_2=1-\epsilon$, $\lambda=\frac{1}{2}-\frac32\epsilon$ for sufficiently small $\epsilon$, and for $2A$ instead of $A$, we get \[\exp\left(iy\mathcal{L}_{c} (\sigma+u)\right)\ll_{\epsilon',\epsilon}Y^{\epsilon'}\] for $c\in \mathcal{Z}^{\mathrm{c}}$, $\norm(c)\leq Y$, $\Re(u)\geq -\frac{\epsilon}{2}$, and $|\Im(u)|\leq (\log Y)^{2A}-2$.
With these bounds in mind, we obtain upper bounds for the integral in \eqref{three} for $\sigma\leq1$ as follows. Observe that 
\begin{eqnarray*}
\label{L3}
\frac{1}{2\pi i} \int_{L_3} \exp\left(iy\mathcal{L}_{c} (\sigma+u)\right) \Gamma(u) X^u du&\ll& X^{-\frac{\epsilon}{2}}Y^{\epsilon'}+X^{-\frac{\epsilon}{2}}Y^{\epsilon'} \int_{1}^{(\log{Y})^A} |\Gamma(-\frac{\epsilon}{2}+it)| dt \nonumber\\
&\ll& X^{-\frac{\epsilon}{2}}Y^{\epsilon'}.
\end{eqnarray*}
By applying Stirling's formula, we get
 \begin{eqnarray*}
 \label{L1}
\frac{1}{2\pi i} \int_{L_1} \exp\left(iy\mathcal{L}_{c} (\sigma+u)\right) \Gamma(u) X^u du&\ll& X^{1-\sigma+\frac{\epsilon}{2}} \int_{(\log{Y})^A}^{\infty} \exp\left( -\frac{\pi}{2}  t\right) t^{\frac{1}{2}-\sigma+\frac{\epsilon}{2}} dt\nonumber\\
&\ll& X^{1-\sigma+\frac{\epsilon}{2}} \exp\left(- \frac{\pi}{4}  (\log{Y})^A\right) (\log{Y})^{A(\frac{1}{2}-\sigma+\frac{\epsilon}{2})} 
 \end{eqnarray*}
and
\begin{eqnarray*}
\label{L_2}
\frac{1}{2\pi i} \int_{L_2} \exp\left(iy\mathcal{L}_{c} (\sigma+u)\right) \Gamma(u) X^u du&\ll& X^{1-\sigma+\frac{\epsilon}{2}} Y^{\epsilon'}\exp\left(- \frac{\pi}{2}  (\log{Y})^A\right) (\log{Y})^{A(\frac{1}{2}-\sigma+\frac{\epsilon}{2})} . \end{eqnarray*}
A bound similar to $L_2$ holds for $L_4$, and a bound similar to $L_1$ holds for $L_5$.
Letting $A>1$ and $X=Y^\eta$ for $\eta>0$ gives
$$\frac{1}{2\pi i} \int_{L_{Y, \epsilon, A}} \exp\left(iy\mathcal{L}_{c} (\sigma+u)\right) \Gamma(u) X^u du \ll  Y^{\epsilon'}X^{-\frac{\epsilon}{2}}.$$
Summing this estimate over $c\in \mathcal{Z}^{\mathrm{c}}$ implies that
\begin{equation}
\label{three-estimate}
(III)\ll Y^{1+\epsilon'} X^{-\frac{\epsilon}{2}}.
\end{equation}

\subsection{Proof of Proposition \ref{newprop}}
\begin{proof}
If $\sigma\leq1$, inserting \eqref{one-estimate}, \eqref{two-estimate}, and \eqref{three-estimate} in \eqref{main2} yields

\begin{align}\label{eqn:finally}
\sum_{c\in \mathcal{C}}^{\star}\exp\left(iy\mathcal{L}_{c} (\sigma)\right)\exp(-\norm(c)/Y)&=\widetilde{C}_{\sigma}(y) Y+O\left(YX^{\epsilon-\epsilon_0}+Y^{\frac{1}{2}+\epsilon}+Y^{\frac12+2\epsilon}X^{1-\epsilon_{0}+3\epsilon}\right) +O\left( Y^\delta X^{\frac12-\epsilon_0+\epsilon}  \right)\\&\hspace{1em}+O\left( Y^{1+\epsilon'} X^{-\frac{\epsilon}{2}}\right)\nonumber.
\end{align}
Now choosing $X=Y^\eta$ for a sufficiently small positive constant $\eta$ controls the contribution of the first and the second error terms in \eqref{eqn:finally} and choosing $0<\epsilon'<\frac{\epsilon\eta}{2}$ controls the contribution of the the third error term in \eqref{eqn:finally}, so that  \[
\sum_{c\in \mathcal{C}}^{\star}\exp\left(iy\mathcal{L}_{c} (\sigma)\right)\exp(-\norm(c)/Y)=\widetilde{C}_{\sigma}(y) Y+o(Y).\]Finally, employing \eqref{N*} yields the result.

If $\sigma>1$, then estimates \eqref{two-estimate} and \eqref{three-estimate} are not needed. In this case we apply Lemma \ref{lem:exp(iyLc)} to write $\exp\left(iy\mathcal{L}_{c} (\sigma+u)\right)$ as an absolutely convergent Dirichlet series (for $u>0$), and we use identity \eqref{eqn:exponential}.  As a result, we get
$$\frac{1}{2\pi i} \int_{(\epsilon)} \exp\left(iy\mathcal{L}_{c} (\sigma+u)\right) \Gamma(u) X^u du=  \sum_{\fa,\fb\subset\mathfrak{O}_{k}}\frac{{\lambda}_{y}(\fa)\lambda_{y}(\fb)\chi_{c}(\fa\fb^{2})}{\norm(\fa)^{{\sigma}}\norm(\fb)^{\sigma}}\exp\left(-\frac{\norm(\fa\fb)}{X}\right),$$ where we choose $\epsilon$ such that $0<\epsilon<\sigma-1$. 
Moving the line of integration from $\Re(u)=\epsilon$ to $\Re(u)=-\epsilon$ and calculating the residue at $u=0$ yield  \begin{align*}
 \sum_{c\in \mathcal{C}}^{\star}\exp\left(iy\mathcal{L}_{c} (\sigma)\right)\exp(-\norm(c)/Y)&=\sum_{c\in \mathcal{C}}\exp\left(iy\mathcal{L}_{c} (\sigma)\right)\exp(-\norm(c)/Y)\\&=\sum _{c\in \mathcal{C}} \left ( \sum_{\fa,\fb\subset\mathfrak{O}_{k}}\frac{{\lambda}_{y}(\fa)\lambda_{y}(\fb)\chi_{c}(\fa\fb^{2})}{\norm(\fa)^{{\sigma}}\norm(\fb)^{\sigma}}\exp\left(-\frac{\norm(\fa\fb)}{X}\right)\right) \exp\left(-\frac{\norm(c)}{Y}\right)\\&\hspace{1em}- \sum _{c\in \mathcal{C}} \left ( \frac{1}{2\pi i} \int_{(-\epsilon)} \exp\left(iy\mathcal{L}_{c} (\sigma+u)\right)\Gamma(u) X^u du
\right) \exp\left(-\frac{\norm(c)}{Y}\right).
\end{align*}
 The result will then follow from \eqref{one-estimate}, Stirling's formula, and choice of $X$ as a positive power of $Y$.
\end{proof}

In what follows, we derive a product formula for $\widetilde{M}_{\sigma}(y)$ which is the absolutely convergent series \[\sum_{r_{1},r_{2}\geq0}\frac{\lambda_{y}(\langle1-\zeta_{3}\rangle^{r_1})\lambda_{y}(\langle1-\zeta_{3}\rangle^{r_2})}{3^{(r_1+r_{2})\sigma}}\sum_{\substack{\fa,\fb,\fm\subset\mathfrak{O}_{k}\\\gcd(\fa\fb\fm,\langle3\rangle)=1\\\gcd(\fa,\fb)=1}}\frac{\lambda_{y}(\fa^3\fm)\lambda_{y}(\fb^3\fm)}{\norm(\fa^3\fb^3\fm^2)^{\sigma}\displaystyle{\prod_{\substack{\fp|\fa\fb\fm\\\fp\; \text{prime}}}\left(1+\norm(\fp)^{-1}\right)}}\]
given in Proposition \ref{newprop}.

\subsection{The product formula for
$\widetilde{M}_{\sigma}(y)$}
\label{sec:prod}

\begin{prop}\label{lem:euler}
In (Case 1) we have 
\begin{align*}%\label{eqn:euler-product-case2}
\widetilde{M}_{\sigma}(y)&=\exp\left(-2iy\log(1-3^{-\sigma})\right)\nonumber\\\hspace{0.1em}&\times\prod_{\fp\nmid\langle3\rangle} \left( \frac{1}{\norm(\fp)+1}+\frac{1}{3}\left(\frac{\norm(\fp)}{\norm(\fp)+1}\right)\sum_{j=0}^{2}\exp\left(-2iy\log\left|1-\frac{\zeta_{3}^{j}}{\norm(\fp)^{\sigma}}\right|\right)\right).\end{align*}
In (Case 2) we have 
\begin{align*}%\label{eqn:euler-product}
\widetilde{M}_{\sigma}(y)&=\exp\left(-2iy\frac{\log 3}{3^{\sigma}-1}\right)\nonumber\\\hspace{0.1em}&\times\prod_{\fp\nmid\langle3\rangle} \left( \frac{1}{\norm(\fp)+1}+\frac{1}{3}\left(\frac{\norm(\fp)}{\norm(\fp)+1}\right)\sum_{j=0}^{2}\exp\left(-2iy\Re\left(\frac{\zeta_{3}^{j} \log\norm(\fp)}{\norm(\fp)^{\sigma}-\zeta_{3}^{j}}\right)\right)\right).
\end{align*}

\end{prop}
\begin{proof}
First notice that in (Case 1) by employing \eqref{H} we have \begin{align*}\sum_{r_{1},r_{2}\geq0}\frac{\lambda_{y}(\langle1-\zeta_{3}\rangle^{r_1})\lambda_{y}(\langle1-\zeta_{3}\rangle^{r_2})}{3^{(r_1+r_{2})\sigma}}&=\sum_{r_{1}\geq0}\frac{\lambda_{y}(\langle1-\zeta_{3}\rangle^{r_1})}{3^{r_1\sigma}}\sum_{r_{2}\geq0}\frac{\lambda_{y}(\langle1-\zeta_{3}\rangle^{r_2})}{3^{r_2\sigma}}\\&=\sum_{r_{1}\geq0}\frac{H_{r_1}(iy)}{3^{r_1\sigma}}\sum_{r_{2}\geq0}\frac{H_{r_2}(iy)}{3^{r_2\sigma}}\\&=\exp\left(-2iy\log(1-3^{-\sigma})\right).\end{align*}
Similarly, in (Case 2) by employing \eqref{G} we get \begin{align*}\sum_{r_{1},r_{2}\geq0}\frac{\lambda_{y}(\langle1-\zeta_{3}\rangle^{r_1})\lambda_{y}(\langle1-\zeta_{3}\rangle^{r_2})}{3^{(r_1+r_{2})\sigma}}&=\exp\left(-2iy\frac{\log 3}{3^{\sigma}-1}\right).\end{align*}
In what follows, $\fp$ denotes a prime ideal that does not divide $\langle3\rangle=3\ringO_{K}$. This specification will be dropped from our notation for simplicity. First we set
\begin{align*}
\widetilde{N}_{\sigma}(y)&:=\sum_{\substack{\fa,\fb,\fm\subset\mathfrak{O}_{k}\\\gcd(\fa\fb\fm,\langle3\rangle)=1\\\gcd(\fa,\fb)=1}}\frac{\lambda_{y}(\fa^3\fm)\lambda_{y}(\fb^3\fm)}{\norm(\fa\fb)^{3\sigma}\norm(\fm)^{2\sigma}}\displaystyle{\prod_{\substack{\fp|\fa\fb\fm\\\fp\; \text{prime}}}\left(1+\norm(\fp)^{-1}\right)^{-1}}.
\end{align*}
Notice that 
\begin{align}
\label{N1}
\widetilde{N}_{\sigma}(y)&=\sum_{\fm\text{ cubefree}}\frac{1}{\norm(\fm)^{2\sigma}}\prod_{\fp|\fm}(1+\norm(\fp)^{-1})^{-1}\\&\hspace{3em}\times\sum_{\fa}\frac{\lambda_{y}(\fa^{3}\fm)}{\norm(\fa)^{3\sigma}}\prod_{\substack{\fp|\fa\\\fp\nmid\fm}}(1+\norm(\fp)^{-1})^{-1}\sum_{\fb}\frac{\lambda_{y}(\fb^{3}\fm)}{\norm(\fb)^{3\sigma}}\prod_{\substack{\fp|\fb\\\fp\nmid\fa\fm}}(1+\norm(\fp)^{-1})^{-1}.\nonumber
\end{align}

Our goal here is to find an Euler product for $\widetilde{N}_{\sigma}(y)$ which is given as the above triple sum over ideals $\fm$, $\fa$, and $\fb$. We start by finding an Euler product for 
the innermost sum in \eqref{N1}. Denoting the multiplicity of a prime ideal $\fp$ in an ideal $\fm$ by $\nu_\fp(\fm)$, we have
\begin{align}
\label{N2}
\sum_{\fb}\frac{\lambda_{y}(\fb^{3}\fm)}{\norm(\fb)^{3\sigma}}\prod_{\substack{\fp|\fb\\\fp\nmid\fa\fm}}(1+\norm(\fp)^{-1})^{-1}&=\prod_{\fp}\left(1+\sum_{j=1}^{\infty}\frac{\lambda_{y}(\fp^{3j})}{\norm(\fp)^{3j\sigma}}\left(1+\norm(\fp)^{-1}\right)^{-1}\right)\\&\hspace{1em}\times\frac{\prod_{\fp|\fa\fm}\left(\sum_{j=0}^{\infty}\frac{\lambda_{y}(\fp^{3j+v_{\fp}(\fm)})}{\norm(\fp)^{3j\sigma}}\right)}{\prod_{\fp|\fa\fm}\left(1+\sum_{j=1}^{\infty}\frac{\lambda_{y}(\fp^{3j})}{\norm(\fp)^{3j\sigma}}\left(1+\norm(\fp)^{-1}\right)^{-1}\right)}.\nonumber
\end{align}

Employing \eqref{N2} in \eqref{N1} yields
\begin{align}
\label{N3}
\widetilde{N}_{\sigma}(y)&=\prod_{\fp}\left(1+\sum_{j=1}^{\infty}\frac{\lambda_{y}(\fp^{3j})}{\norm(\fp)^{3j\sigma}}\left(1+\norm(\fp)^{-1}\right)^{-1}\right)\\&\hspace{2em}\times\sum_{\fm\text{ cubefree}}\frac{1}{\norm(\fm)^{2\sigma}}\prod_{\fp|\fm}(1+\norm(\fp)^{-1})^{-1}\sum_{\fa}\frac{\lambda_{y}(\fa^{3}\fm)}{\norm(\fa)^{3\sigma}}\nonumber
\\&\hspace{2em}\times\prod_{\substack{\fp|\fa\\\fp\nmid\fm}}(1+\norm(\fp)^{-1})^{-1}\frac{\prod_{\fp|\fa\fm}\left(\sum_{j=0}^{\infty}\frac{\lambda_{y}(\fp^{3j+v_{\fp}(\fm)})}{\norm(\fp)^{3j\sigma}}\right)}{\prod_{\fp|\fa\fm}\left(1+\sum_{j=1}^{\infty}\frac{\lambda_{y}(\fp^{3j})}{\norm(\fp)^{3j\sigma}}\left(1+\norm(\fp)^{-1}\right)^{-1}\right)}.\nonumber\end{align}

Let us now focus in above on the sum over cube-free integral ideals $\fm$ which we denote by $\widetilde{P}_{\sigma}(y)$. We have 
\begin{align*} \widetilde{P}_{\sigma}(y)&=\sum_{\fm\text{ cubefree}}\frac{1}{\norm(\fm)^{2\sigma}}\prod_{\fp|\fm}\frac{\left(\sum_{j=0}^{\infty}\frac{\lambda_{y}(\fp^{3j+v_{\fp}(\fm)})}{\norm(\fp)^{3j\sigma}}(1+\norm(\fp)^{-1})^{-1}\right)}{\left(1+\sum_{j=1}^{\infty}\frac{\lambda_{y}(\fp^{3j})}{\norm(\fp)^{3j\sigma}}\left(1+\norm(\fp)^{-1}\right)^{-1}\right)}\\&\hspace{1em}\times\sum_{\fa}\frac{\lambda_{y}(\fa^{3}\fm)}{\norm(\fa)^{3\sigma}}\prod_{\substack{\fp|\fa\\\fp\nmid\fm}}\frac{\left(\sum_{j=0}^{\infty}\frac{\lambda_{y}(\fp^{3j+v_{\fp}(\fm)})}{\norm(\fp)^{3j\sigma}}\left(1+\norm(\fp)^{-1}\right)^{-1}\right)}{\left(1+\sum_{j=1}^{\infty}\frac{\lambda_{y}(\fp^{3j})}{\norm(\fp)^{3j\sigma}}\left(1+\norm(\fp)^{-1}\right)^{-1}\right)}.
\end{align*}
Writing the sum over $\fa$ as an Euler product yields
\begin{align*}
\widetilde{P}_{\sigma}(y)&=\sum_{\fm\text{ cubefree}}\frac{1}{\norm(\fm)^{2\sigma}}\prod_{\fp|\fm}\frac{\left(\sum_{j=0}^{\infty}\frac{\lambda_{y}(\fp^{3j+v_{\fp}(\fm)})}{\norm(\fp)^{3j\sigma}}(1+\norm(\fp)^{-1})^{-1}\right)}{\left(1+\sum_{j=1}^{\infty}\frac{\lambda_{y}(\fp^{3j})}{\norm(\fp)^{3j\sigma}}\left(1+\norm(\fp)^{-1}\right)^{-1}\right)}\\&\hspace{1em}\times\prod_{\fp\nmid\fm}\left(1+\sum_{k=1}^{\infty}\frac{\lambda_{y}(\fp^{3k})}{\norm(\fp)^{3k\sigma}}\frac{\left(\sum_{j=0}^{\infty}\frac{\lambda_{y}(\fp^{3j})}{\norm(\fp)^{3j\sigma}}\left(1+\norm(\fp)^{-1}\right)^{-1}\right)}{\left(1+\sum_{j=1}^{\infty}\frac{\lambda_{y}(\fp^{3j})}{\norm(\fp)^{3j\sigma}}\left(1+\norm(\fp)^{-1}\right)^{-1}\right)}\right)\prod_{\fp|\fm}\left(\sum_{k=0}^{\infty}\frac{\lambda_{y}(\fp^{3k+v_{\fp}(\fm)})}{\norm(\fp)^{3k\sigma}}\right).
\end{align*}
Thus, we have
\begin{align}
\label{N4}
\widetilde{P}_{\sigma}(y)&=\prod_{\fp}\left(1+\sum_{k=1}^{\infty}\frac{\lambda_{y}(\fp^{3k})}{\norm(\fp)^{3k\sigma}}\frac{\left(\sum_{j=0}^{\infty}\frac{\lambda_{y}(\fp^{3j})}{\norm(\fp)^{3j\sigma}}\left(1+\norm(\fp)^{-1}\right)^{-1}\right)}{\left(1+\sum_{j=1}^{\infty}\frac{\lambda_{y}(\fp^{3j})}{\norm(\fp)^{3j\sigma}}\left(1+\norm(\fp)^{-1}\right)^{-1}\right)}\right)\\
&\hspace{1em}\times\sum_{\fm\text{ cubefree}}\frac{1}{\norm(\fm)^{2\sigma}}\frac{\prod_{\fp|\fm}\frac{\left(\sum_{j=0}^{\infty}\frac{\lambda_{y}(\fp^{3j+v_{\fp}(\fm)})}{\norm(\fp)^{3j\sigma}}(1+\norm(\fp)^{-1})^{-1}\right)}{\left(1+\sum_{j=1}^{\infty}\frac{\lambda_{y}(\fp^{3j})}{\norm(\fp)^{3j\sigma}}\left(1+\norm(\fp)^{-1}\right)^{-1}\right)}\prod_{\fp|\fm}\left(\sum_{k=0}^{\infty}\frac{\lambda_{y}(\fp^{3k+v_{\fp}(\fm)})}{\norm(\fp)^{3k\sigma}}\right)}{\prod_{\fp|\fm}\left(1+\sum_{k=1}^{\infty}\frac{\lambda_{y}(\fp^{3k})}{\norm(\fp)^{3k\sigma}}\frac{\left(\sum_{j=0}^{\infty}\frac{\lambda_{y}(\fp^{3j})}{\norm(\fp)^{3j\sigma}}\left(1+\norm(\fp)^{-1}\right)^{-1}\right)}{\left(1+\sum_{j=1}^{\infty}\frac{\lambda_{y}(\fp^{3j})}{\norm(\fp)^{3j\sigma}}\left(1+\norm(\fp)^{-1}\right)^{-1}\right)}\right)}.\nonumber
\end{align}

Next observe that the new $\fm$-summation in the above expression has the following Euler product decomposition:
\begin{align}
\label{N5}
\prod_{\fp}\left(1+\sum_{i=1}^{2}\frac{1}{\norm(\fp)^{2i\sigma}}\frac{\frac{\left(\sum_{j=0}^{\infty}\frac{\lambda_{y}(\fp^{3j+i})}{\norm(\fp)^{3j\sigma}}(1+\norm(\fp)^{-1})^{-1}\right)}{\left(1+\sum_{j=1}^{\infty}\frac{\lambda_{y}(\fp^{3j})}{\norm(\fp)^{3j\sigma}}\left(1+\norm(\fp)^{-1}\right)^{-1}\right)}\left(\sum_{k=0}^{\infty}\frac{\lambda_{y}(\fp^{3k+i})}{\norm(\fp)^{3k\sigma}}\right)}{\left(1+\sum_{k=1}^{\infty}\frac{\lambda_{y}(\fp^{3k})}{\norm(\fp)^{3k\sigma}}\frac{\left(\sum_{j=0}^{\infty}\frac{\lambda_{y}(\fp^{3j})}{\norm(\fp)^{3j\sigma}}\left(1+\norm(\fp)^{-1}\right)^{-1}\right)}{\left(1+\sum_{j=1}^{\infty}\frac{\lambda_{y}(\fp^{3j})}{\norm(\fp)^{3j\sigma}}\left(1+\norm(\fp)^{-1}\right)^{-1}\right)}\right)}\right).
\end{align}

Employing \eqref{N5} in \eqref{N4} and inserting the deduced result for 
$\widetilde{P}_{\sigma}(y)$ in \eqref{N3} yield
\begin{equation*}
\widetilde{N}_{\sigma}(y)=\prod_{\fp}\widetilde{M}_{\sigma,\fp}(y),
\end{equation*}
where
\begin{align*}
\widetilde{M}_{\sigma,\fp}(y)&=1-\left(1+\norm(\fp)^{-1}\right)^{-1}
+\left(1+\norm(\fp)^{-1}\right)^{-1}\sum_{i=0}^{2}\sum_{\substack{j=0\\k=0}}^{\infty}\frac{\lambda_{y}(\fp^{3j+i})\lambda_{y}(\fp^{3k+i})}{\norm(\fp)^{(3j+i)\sigma}\norm(\fp)^{(3k+i)\sigma}}.
\end{align*}

Using the definition of the arithmetic function $\lambda_{y}$ in (Case 2), we get
\begin{align*}
\sum_{i=0}^{2}\sum_{\substack{j=0\\k=0}}^{\infty}\frac{\lambda_{y}(\fp^{3j+i})\lambda_{y}(\fp^{3k+i})}{\norm(\fp)^{(3j+i)\sigma}\norm(\fp)^{(3k+i)\sigma}}&=\sum_{i=0}^{2}\sum_{j=0}^{\infty}\frac{G_{3j+i}\left(-iy\log\norm(\fp)\right)}{\norm(\fp)^{(3j+i)s}}\sum_{k=0}^{\infty}\frac{G_{3k+i}\left(-iy\log\norm(\fp)\right)}{\norm(\fp)^{(3k+i)\sigma}},
\end{align*}
which by 
\eqref{G}
is
\begin{align*}
=\frac19\sum_{i=0}^{2}\sum_{\substack{n=0\\m=0}}^{2}\frac{1}{\zeta_{3}^{i(n+m)}}\exp\left(-i\log\norm(\fp)\left(\frac{y\zeta_{3}^{n}}{\norm(\fp)^{\sigma}-\zeta_{3}^{n}}+\frac{y\zeta_{3}^{m}}{\norm(\fp)^{\sigma}-\zeta_{3}^{m}}\right)\right).
\end{align*}
Since 
\[\sum_{i=0}^{2}\frac{1}{\zeta_{3}^{i(n+m)}}=\begin{cases}
3 & \text{ if }n+m=0\text{ or }3, \\
0 & \text{ otherwise}, 
 \end{cases}\]
we conclude that 
\begin{align*}
\sum_{i=0}^{2}\sum_{\substack{j=0\\k=0}}^{\infty}\frac{\lambda_{y}(\fp^{3j+i})\lambda_{y}(\fp^{3k+i})}{\norm(\fp)^{(3j+i)\sigma}\norm(\fp)^{(3k+i)\sigma}}&=\frac13\sum_{j=0}^{2}\exp\left(-2iy\Re\left(\frac{\zeta_{3}^{j} \log\norm(\fp)}{\norm(\fp)^{\sigma}-\zeta_{3}^{j}}\right)\right).
\end{align*}
Therefore, $\widetilde{N}_{\sigma}(y)$ in (Case 2) can be written as
\begin{align*}
\widetilde{N}_{\sigma}(y)&=\prod_{\fp} \left(\frac{1}{\norm(\fp)+1}+\frac{1}{3}\left(\frac{\norm(\fp)}{\norm(\fp)+1}\right)\sum_{j=0}^{2}\exp\left(-2iy\Re\left(\frac{\zeta_{3}^{j}\log\norm(\fp)}{\norm(\fp)^{\sigma}-\zeta_{3}^{j}}\right)\right)\right).
\end{align*}
Similar calculations show that $\widetilde{N}_{\sigma}(y)$ in (Case 1) takes the form
\begin{align*}
\widetilde{N}_{\sigma}(y)&=\prod_{\fp} \left( \frac{1}{\norm(\fp)+1}+\frac{1}{3}\left(\frac{\norm(\fp)}{\norm(\fp)+1}\right)\sum_{j=0}^{2}\exp\left(-2iy\log\left|1-\frac{\zeta_{3}^{j}}{\norm(\fp)^{\sigma}}\right|\right)\right).
\end{align*}

\end{proof}

\section{PROOF OF PROPOSITION  \ref{mainprop2}}\label{sec:good-density}
\begin{proof}
Notice that $\overline{\widetilde{M}_{\sigma}(y)}=\widetilde{M}_{\sigma}(-y)$ since $\overline{\lambda_{y}(\fa)}=\lambda_{-y}(\fa)$. Thus, we may assume without loss of generality that $y>0$. First, we deal with (Case 1). By Lemma \ref{lem:euler}, we have 
\begin{equation*}
\widetilde{M}_{\sigma}(y)=\exp\left(-2iy\log(1-3^{-\sigma})\right)\prod_{\fp\nmid\langle3\rangle}\widetilde{M}_{\sigma,\fp}(y),
\end{equation*}
where
\begin{equation*} \widetilde{M}_{\sigma,\fp}(y)=  \frac{1}{\norm(\fp)+1}+\frac{1}{3}\left(\frac{\norm(\fp)}{\norm(\fp)+1}\right)\sum_{j=0}^{2} \exp\left(-2iy\log\left|1-\frac{\zeta_{3}^{j}}{\norm(\fp)^{\sigma}}\right|\right).
\end{equation*}

Note that for all $y$ we have $|\widetilde{M}_{\sigma,\fp}(y)|\leq 1$. Our goal here is to show that for large $y$ a sizeable number of $\widetilde{M}_{\sigma,\fp}(y)$ have absolute value less than $1-\beta$ for a fixed $\beta>0$, more precisely we will show that the number of such $\widetilde{M}_{\sigma,\fp}(y)$ is bounded below by a power of $y$ (depending on $\sigma$).

Notice that 
\begin{align*}
\widetilde{Q}_{\sigma,\fp}(y)&=\sum_{j=0}^{2}\exp\left(-2iy\log\left|1-\frac{\zeta_{3}^{j}}{\norm(\fp)^{\sigma}}\right|\right).
\\&=\exp\left(-2iy\log(1-p^{-\sigma})\right)\left(1+2\exp\left(-2iy\log\frac{\sqrt{\norm(\fp)^{2\sigma}+\norm(\fp)^{\sigma}+1}}{\norm(\fp)^{\sigma}-1}\right)\right),
\end{align*}
and so 
\begin{align*}
\left|\widetilde{Q}_{\sigma,\fp}(y)\right|&=\left|1+2\exp\left(-2iy\log\frac{\sqrt{\norm(\fp)^{2\sigma}+\norm(\fp)^{\sigma}+1}}{\norm(\fp)^{\sigma}-1}\right)\right|\\&=\left|1+2\cos\left(2y\log\frac{\sqrt{\norm(\fp)^{2\sigma}+\norm(\fp)^{\sigma}+1}}{\norm(\fp)^{\sigma}-1}\right)-2i\sin\left(2y\log\frac{\sqrt{\norm(\fp)^{2\sigma}+\norm(\fp)^{\sigma}+1}}{\norm(\fp)^{\sigma}-1}\right)\right|\\&\leq2\cos^{2}\left(y\log\frac{\sqrt{\norm(\fp)^{2\sigma}+\norm(\fp)^{\sigma}+1}}{\norm(\fp)^{\sigma}-1}\right)+4\left|\cos\left(y\log\frac{\sqrt{\norm(\fp)^{2\sigma}+\norm(\fp)^{\sigma}+1}}{\norm(\fp)^{\sigma}-1}\right)\right|.
\end{align*}
For any $\epsilon>0$ and $y$ large enough, we consider the prime ideals $\fp$ for which the  condition
\begin{equation}\label{eqn:condition2}
1.35-\epsilon\leq y\log\frac{\sqrt{\norm(\fp)^{2\sigma}+\norm(\fp)^{\sigma}+1}}{\norm(\fp)^{\sigma}-1}\leq 1.77+\epsilon
\end{equation}
holds.
 By taking $\epsilon$ small enough, we get $\left|\cos\left(y\log\frac{\sqrt{\norm(\fp)^{2\sigma}+\norm(\fp)^{\sigma}+1}}{\norm(\fp)^{\sigma}-1}\right) \right|\leq 0.22$ which implies that $\left|\widetilde{Q}_{\sigma,\fp}(y)\right|\leq 0.9768$. Hence, for all $\fp$ satisfying (\ref{eqn:condition2}), we have \[\left|\widetilde{M}_{\sigma,\fp}(y)\right|\leq \frac{1}{\norm(\fp)+1}+0.3256\left(\frac{\norm(\fp)}{\norm(\fp)+1}\right)\leq 0.8256.\]
Let $\Pi(x)$ be the number of prime ideals in $k$ with norm at most $x$, and let $\Pi_{\sigma}(y)$ be the number of prime ideals satisfying (\ref{eqn:condition2}). Notice that 
$\frac{2y}{2.36}\leq \norm(\fp)^{\sigma}\leq \frac{2y}{1.8}$ is equivalent to 
 \[ \frac{1.8}{2}\norm(\fp)^{\sigma}\log\tfrac{\sqrt{\norm(\fp)^{2\sigma}+\norm(\fp)^{\sigma}+1}}{\norm(\fp)^{\sigma}-1}\leq y\log\tfrac{\sqrt{\norm(\fp)^{2\sigma}+\norm(\fp)^{\sigma}+1}}{\norm(\fp)^{\sigma}-1}\leq \frac{2.36}{2}\norm(\fp)^{\sigma}\log\tfrac{\sqrt{\norm(\fp)^{2\sigma}+\norm(\fp)^{\sigma}+1}}{\norm(\fp)^{\sigma}-1}.\]Since 
 $$\displaystyle{\lim_{\norm(\fp)\to\infty}\norm(\fp)^{\sigma}\log\tfrac{\sqrt{\norm(\fp)^{2\sigma}+\norm(\fp)^{\sigma}+1}}{\norm(\fp)^{\sigma}-1}=\frac32},$$ 
 we get that \[\frac{2y}{2.36}\leq \norm(\fp)^{\sigma}\leq \frac{2y}{1.8}\implies 1.35-\epsilon\leq y\log\frac{\sqrt{\norm(\fp)^{2\sigma}+\norm(\fp)^{\sigma}+1}}{\norm(\fp)^{\sigma}-1}\leq 1.77 +\epsilon.\]
It follows that \[\Pi_{\sigma}(y)>\Pi\left(\left(\frac{2y}{1.8}\right)^{\frac{1}{\sigma}}\right)-\Pi\left(\left(\frac{2y}{2.36}\right)^{\frac{1}{\sigma}}\right)\gg_{\sigma} y^{\frac{1}{\sigma}-\delta}\] for any sufficiently small $\delta>0$. 
Therefore, 
\[\left|\widetilde{M}_{\sigma}(y)\right|=\prod_{\fp}\left|\widetilde{M}_{\sigma,\fp}(y)\right|\leq 0.8256^{\Pi_{\sigma}(y)}=\exp\left(-\log\left(\frac{1}{0.8256}\right)\Pi_{\sigma}(y)\right)\leq\exp\left(-Cy^{\frac{1}{\sigma}-\delta}\right),\] where $C$ is a positive constant depending only on $\sigma$ and $\delta$.

Next, we deal with (Case 2).  By Lemma \ref{lem:euler}, we have 
\begin{equation*}
\widetilde{M}_{\sigma}(y)=\exp\left(-2iy\frac{\log 3}{3^{\sigma}-1}\right)\prod_{\fp\nmid\langle3\rangle}\widetilde{M}_{\sigma,\fp}(y),
\end{equation*}
where
\begin{equation*} \widetilde{M}_{\sigma,\fp}(y)=\frac{1}{\norm(\fp)+1}+\frac{1}{3}\left(\frac{\norm(\fp)}{\norm(\fp)+1}\right)\sum_{j=0}^{2}\exp\left(-2iy\log\norm(\fp)\Re\left(\frac{\zeta_{3}^{j}}{\norm(\fp)^{\sigma}-\zeta_{3}^{j}}\right)\right).
\end{equation*}

Let us now consider the finite sum \[\widetilde{Q}_{\sigma,\fp}(y)=\sum_{j=0}^{2}\exp\left(-2iy\log\norm(\fp)\Re\left(\frac{\zeta_{3}^{j}}{\norm(\fp)^{\sigma}-\zeta_{3}^{j}}\right)\right).\] 
For ease of notation, we put $\alpha_{j}=2y\log\norm(\fp)\Re\left(\frac{\zeta_{3}^{j}}{\norm(\fp)^{\sigma}-\zeta_{3}^{j}}\right)$ for $j=0,1,2$. Since $\alpha_{1}=\alpha_{2}$, we have 
\begin{align*}\widetilde{Q}_{\sigma,\fp}(y)&=\exp\left(-i\alpha_{0}\right)+2\exp\left(-i\alpha_{1}\right)\\&=\exp\left(-i\alpha_{0}\right)\left[1+2\exp\left(-i\alpha_{1}+i\alpha_{0}\right)\right].\end{align*}
Therefore,
\begin{align*}|\widetilde{Q}_{\sigma,\fp}(y)|&=\left|1+2\exp\left(i\left(\alpha_{0}-\alpha_{1}\right)\right)\right|\\&=\left|1+2\cos(\alpha_{0}-\alpha_{1})+2i\sin(\alpha_{0}-\alpha_{1})\right|\\&=\left|2\cos^{2}\left(\frac{\alpha_{0}-\alpha_{1}}{2}\right)+4i\sin\left(\frac{\alpha_{0}-\alpha_{1}}{2}\right)\cos\left(\frac{\alpha_{0}-\alpha_{1}}{2}\right)\right|\\&\leq 2\left(\cos^{2}\left(\frac{\alpha_{0}-\alpha_{1}}{2}\right)+2\left|\cos\left(\frac{\alpha_{0}-\alpha_{1}}{2}\right)\right|\right).
\end{align*}
For any $\epsilon>0$ and $y$ large enough, we consider the prime ideals $\fp$ for which the condition 
\begin{equation}\label{eqn:condition1}
1.35-\epsilon\leq\frac{\alpha_{0}-\alpha_{1}}{2}\leq 1.77+\epsilon
\end{equation}
holds. Once again, by taking $\epsilon$ small enough, we get $\left|\cos\left(\frac{\alpha_{0}-\alpha_{1}}{2}\right)\right|\leq 0.22$ which implies that $\left|\widetilde{Q}_{\sigma,\fp}(y)\right|\leq 0.9768$. Hence, for all $\fp$ satisfying (\ref{eqn:condition1}), we have \[\left|\widetilde{M}_{\sigma,\fp}(y)\right|\leq \frac{1}{\norm(\fp)+1}+0.3256\left(\frac{\norm(\fp)}{\norm(\fp)+1}\right)\leq 0.8256.\]
 Notice that 
\[\frac{2y\log 2y}{2.36\sigma}\leq \norm(\fp)^{\sigma}\leq \frac{2y\log 2y}{1.8\sigma}\implies 1.35-\epsilon\leq\frac{\alpha_{0}-\alpha_{1}}{2}\leq 1.77+\epsilon.\]
We deduce that the number of prime ideals satisfying (\ref{eqn:condition1}) is at least \[\Pi\left(\left(\frac{2y\log 2y}{1.8\sigma}\right)^{\frac{1}{\sigma}}\right)-\Pi\left(\left(\frac{2y\log 2y}{2.36\sigma}\right)^{\frac{1}{\sigma}}\right)\gg_{\sigma} y^{\frac{1}{\sigma}}.\] 
Therefore, 
\[\left|\widetilde{M}_{\sigma}(y)\right|\leq\exp\left(-C'y^{\frac{1}{\sigma}}\right),\] where $C'$ is a positive constant depending only on $\sigma$.

\end{proof}

\begin{rezabib}

\bib{Bill}{book} {
author={Billingsley, Patrick},
   title={Probability and measure},
   series={Wiley Series in Probability and Mathematical Statistics},
   edition={3},
   note={A Wiley-Interscience Publication},
   publisher={John Wiley \& Sons, Inc., New York},
   date={1995},
   pages={xiv+593},

}

\bib{BGL}{article} {
 author={Blomer, Valentin},
   author={Goldmakher, Leo},
   author={Louvel, Beno\^{i}t},
   title={$L$-functions with $n$-th-order twists},
   journal={Int. Math. Res. Not. IMRN},
   date={2014},
   number={7},
   pages={1925--1955},
 
}
	   
 \bib{CK}{article}{
   author={Cho, Peter J.},
   author={Kim, Henry H.},
   title={Moments of logarithmic derivatives of $L$-functions},
   journal={J. Number Theory},
   volume={183},
   date={2018},
   pages={40--61},

}

\bib{chowla-erdos}{article} {
   author={Chowla, S.},
   author={Erd\"{o}s, P.},
   title={A theorem on the distribution of the values of $L$-functions},
   journal={J. Indian Math. Soc. (N.S.)},
   volume={15},
   date={1951},
   pages={11--18},
 
}

\bib{elliott-0}{article}{
   author={Elliott, P. D. T. A.},
   title={The distribution of the quadratic class number},
   language={English, with Lithuanian and Russian summaries},
   journal={Litovsk. Mat. Sb.},
   volume={10},
   date={1970},
   pages={189--197},

}

\bib{elliott-01}{article}{
   author={Elliott, P. D. T. A.},
   title={On the distribution of the values of Dirichlet $L$-series in the
   half-plane $\sigma >{1\over 2}$},
   journal={Nederl. Akad. Wetensch. Proc. Ser. A {\bf 74}=Indag. Math.},
   volume={33},
   date={1971},
   pages={222--234},
  
}
	
\bib{elliott-02}{article}{
   author={Elliott, P. D. T. A.},
   title={On the distribution of ${\rm arg}L(s,\,\chi )$ in the half-plane
   $\sigma >{1\over 2}$},
   journal={Acta Arith.},
   volume={20},
   date={1972},
   pages={155--169},
  
}

\bib{elliott}{article}{
   author={Elliott, P. D. T. A.},
   title={On the distribution of the values of quadratic $L$-series in the
   half-plane $\sigma >{1\over 2}$},
   journal={Invent. Math.},
   volume={21},
   date={1973},
   pages={319--338},
 
}

\bib{elliott-book-I}{book}{
   author={Elliott, P. D. T. A.},
   title={Probabilistic number theory. I},
   series={Grundlehren der Mathematischen Wissenschaften [Fundamental
   Principles of Mathematical Science]},
   volume={239},
   note={Mean-value theorems},
   publisher={Springer-Verlag, New York-Berlin},
   date={1979},
   pages={xxii+359+xxxiii pp. (2 plates)},

}

\bib{elliott-book}{book}{
   author={Elliott, P. D. T. A.},
   title={Probabilistic number theory. II},
   series={Grundlehren der Mathematischen Wissenschaften [Fundamental
   Principles of Mathematical Science]},
   volume={240},
   note={Central limit theorems},
   publisher={Springer-Verlag, New York-Berlin},
   date={1980},
   pages={xviii+341+xxxiv pp.},
 
}

\bib{Folland}{book} {
  author={Folland, Gerald B.},
   title={Real analysis},
   series={Pure and Applied Mathematics (New York)},
   edition={2},
   note={Modern techniques and their applications;
   A Wiley-Interscience Publication},
   publisher={John Wiley \& Sons, Inc., New York},
   date={1999},
   pages={xvi+386},
  
}

\bib{GS}{article}{
   author={Granville, A.},
   author={Soundararajan, K.},
   title={The distribution of values of $L(1,\chi_d)$},
   journal={Geom. Funct. Anal.},
   volume={13},
   date={2003},
   number={5},
   pages={992--1028},

}

\bib{H}{book}{
    author={Hardy, G. H.},
   title={Divergent Series},
   publisher={Oxford, at the Clarendon Press},
   date={1949},
   pages={xvi+396},
}
	
\bib{MR1670215}{article} {
      author={Hattori, Tetsuya},
   author={Matsumoto, Kohji},
   title={A limit theorem for Bohr-Jessen's probability measures of the
   Riemann zeta-function},
   journal={J. Reine Angew. Math.},
   volume={507},
   date={1999},
   pages={219--232},
}
\bib{HB}{article}{
   author={Heath-Brown, D. R.},
   title={Kummer's conjecture for cubic Gauss sums},
   journal={Israel J. Math.},
   volume={120},
   date={2000},
   number={part A},
   part={part A},
   pages={97--124}, 
}
	
\bib{HB-P}{article}{
   author={Heath-Brown, D. R.},
   author={Patterson, S. J.},
   title={The distribution of Kummer sums at prime arguments},
   journal={J. Reine Angew. Math.},
   volume={310},
   date={1979},
   pages={111--130}, 
   }

\bib{I-M1}{article}{
   author={Ihara, Yasutaka},
   author={Matsumoto, Kohji},
   title={On certain mean values and the value-distribution of logarithms of
   Dirichlet $L$-functions},
   journal={Q. J. Math.},
   volume={62},
   date={2011},
   number={3},
   pages={637--677},

}

\bib{I-M}{article}{
   author={Ihara, Yasutaka},
   author={Matsumoto, Kohji},
   title={On $\log L$ and $L'/L$ for $L$-functions and the associated
   ``$M$-functions'': connections in optimal cases},
   language={English, with English and Russian summaries},
   journal={Mosc. Math. J.},
   volume={11},
   date={2011},
   number={1},
   pages={73--111, 182},}

%\bib{I-M}{article}{
%   author={Ihara, Y.},
%   author={Matsumoto, K.},
%   title={On $\log L$ and $L'/L$ for $L$-functions and the associated
%   ``$M$-functions'': connections in optimal cases},
%   language={English, with English and Russian summaries},
%   journal={Mosc. Math. J.},
%   volume={11},
%   date={2011},
%   number={1},
%   pages={73--111, 182},
% 
%}
%

\bib{I-R}{book}{
   author={Ireland, Kenneth},
   author={Rosen, Michael},
   title={A classical introduction to modern number theory},
   series={Graduate Texts in Mathematics},
   volume={84},
   edition={2},
   publisher={Springer-Verlag, New York},
   date={1990},
   pages={xiv+389},

}

%\bib{Jutila2}{article}{
%   author={Jutila, M.},
%   title={On the mean value of $L({1\over 2},\,\chi )$\ for real characters},
%   journal={Analysis},
%   volume={1},
%   date={1981},
%   number={2},
%   pages={149--161},
%   issn={0174-4747},
%   review={\MR{632705}},
%}
%	
   
%\bib{lamzouri1}{article}{
%   author={Lamzouri, Youness},
%   title={The two-dimensional distribution of values of $\zeta(1+it)$},
 %  journal={Int. Math. Res. Not. IMRN},
 %  date={2008},
  % pages={Art. ID rnn 106, 48},
 %  issn={1073-7928},
 %  review={\MR{2439537}},
%}

\bib{lamzouri2}{article}{
   author={Lamzouri, Youness},
   title={Distribution of values of $L$-functions at the edge of the
   critical strip},
   journal={Proc. Lond. Math. Soc. (3)},
   volume={100},
   date={2010},
   number={3},
   pages={835--863},
 
}

\bib{lamzouri3}{article}{
   author={Lamzouri, Youness},
   title={On the distribution of extreme values of zeta and $L$-functions in
   the strip $\frac12<\sigma<1$},
   journal={Int. Math. Res. Not. IMRN},
   date={2011},
   number={23},
   pages={5449--5503},
 
}

\bib{MR3378382}{article}{
   author={Lamzouri, Youness},
   title={The distribution of Euler-Kronecker constants of quadratic fields},
   journal={J. Math. Anal. Appl.},
   volume={432},
   date={2015},
   number={2},
   pages={632--653},

}

%\bib{lamzouri4}{article}{
 %  author={Lamzouri, Youness},
 %  title={Large values of $L(1,\chi)$ for $k$th order characters $\chi$ and
%   applications to character sums},
%   journal={Mathematika},
 %  volume={63},
%   date={2017},
 %  number={1},
 %  pages={53--71},
 %  issn={0025-5793},
 %  review={\MR{3610005}},
%}

\bib{lemmermeyer}{book}{
    author={Lemmermeyer, Franz},
   title={Reciprocity laws},
   series={Springer Monographs in Mathematics},
   note={From Euler to Eisenstein},
   publisher={Springer-Verlag, Berlin},
   date={2000},
   pages={xx+487},
 
}
	
\bib{luo}{article}{
  author={Luo, Wenzhi},
   title={Values of symmetric square $L$-functions at $1$},
   journal={J. Reine Angew. Math.},
   volume={506},
   date={1999},
   pages={215--235},
} 

\bib{L}{article}{
  author={Luo, Wenzhi},
   title={On Hecke $L$-series associated with cubic characters},
   journal={Compos. Math.},
   volume={140},
   date={2004},
   number={5},
   pages={1191--1196},
} 
	
\bib{M-M}{article}{
   author={Mourtada, Mariam},
   author={Murty, V. Kumar},
   title={Distribution of values of $L'/L(\sigma,\chi_D)$},
   language={English, with English and Russian summaries},
   journal={Mosc. Math. J.},
   volume={15},
   date={2015},
   number={3},
   pages={497--509, 605},
 
}

%   
%   \bib{M-M1}{article}{
%   author={Mourtada, Mariam},
%   author={Kumar Murty, V.},
%   title={The distribution of values of logarithmic derivatives of quadratic
%   Dirichlet $L$-functions},
%   conference={
%      title={Highly composite: papers in number theory},
%   },
%   book={
%      series={Ramanujan Math. Soc. Lect. Notes Ser.},
%      volume={23},
%      publisher={Ramanujan Math. Soc., Mysore},
%   },
%   date={2016},
%   pages={37--50},
%   review={\MR{3692723}},
%}

\bib{T}{book}{
   author={Tenenbaum, G\'{e}rald},
   title={Introduction to analytic and probabilistic number theory},
   series={Graduate Studies in Mathematics},
   volume={163},
   edition={3},
   note={Translated from the 2008 French edition by Patrick D. F. Ion},
   publisher={American Mathematical Society, Providence, RI},
   date={2015},
   pages={xxiv+629},

}
\bib{MR3728294}{book}{
   author={Titchmarsh, E. C.},
   title={The theory of functions},
   edition={2},
   publisher={Oxford University Press, Oxford},
   date={1939},
   pages={x+454},
   
}
	
\bib{Wintner}{article}{
   author={Wintner, Aurel},
   title={On symmetric Bernoulli convolutions},
   journal={Bull. Amer. Math. Soc.},
   volume={41},
   date={1935},
   number={2},
   pages={137--138},

}

\bib{X}{article}{
  author={Xia, Honggang},
   title={On zeros of cubic $L$-functions},
   journal={J. Number Theory},
   volume={124},
   date={2007},
   number={2},
   pages={415--428},
}

  \end{rezabib}

\end{document}